\documentclass[10pt,a4paper]{article} 
\usepackage[final]{optional}
\usepackage[british,american]{babel}

\usepackage{mathrsfs}

\usepackage{amsfonts, amsmath, wasysym}

\usepackage{amssymb,amsthm,
paralist
}

\usepackage{
latexsym,
nicefrac,
}

\usepackage[usenames]{color}

\usepackage{url}

\definecolor{darkgreen}{rgb}{0,0.5,0}
\definecolor{darkred}{rgb}{0.7,0,0}
\usepackage[colorlinks, 
citecolor=darkgreen, linkcolor=darkred
]{hyperref}
\usepackage{esint}
\usepackage{bibgerm}
\usepackage[normalem]{ulem}

\theoremstyle{plain}
\newtheorem{lemma}{Lemma}[section]
\newtheorem{thm}[lemma]{Theorem}
\newtheorem{prop}[lemma]{Proposition}
\newtheorem{cor}[lemma]{Corollary}

\theoremstyle{definition}

\newtheorem{rmk}[lemma]{Remark}

\numberwithin{equation}{section}

\newcommand{\m}{\mathcal{M}}

\newcommand{\stw}{\rho^{-2}\vph^2}

\newcommand{\pl}[2]{{\frac{\partial #1}{\partial #2}}}

\newcommand{\al}{\alpha}
\newcommand{\be}{\beta}
\newcommand{\ga}{\gamma}

\newcommand{\de}{\delta}
\newcommand{\om}{\omega}
\newcommand{\Om}{\Omega}

\newcommand{\la}{\lambda}
\newcommand{\La}{\Lambda}

\newcommand{\si}{\sigma}

\newcommand{\Si}{\Sigma}
\renewcommand{\th}{\theta}
\newcommand{\Th}{\Theta}

\newcommand{\vph}{\varphi}
\newcommand{\vth}{\vartheta}
\newcommand{\ep}{\varepsilon}

\newcommand{\R}{\ensuremath{{\mathbb R}}}
\newcommand{\N}{\ensuremath{{\mathbb N}}}

\newcommand{\C}{\ensuremath{{\mathbb C}}}

\newcommand{\downto}{\downarrow}

\newcommand{\embed}{\hookrightarrow}

\newcommand{\grad}{\nabla}

\newcommand{\intersect}{\cap}
\newcommand{\union}{\cup}

\DeclareMathOperator{\inj}{inj}

\newcommand{\norm}[1]{\Vert#1\Vert}  

\newcommand{\beq}{\begin{equation}}
\newcommand{\eeq}{\end{equation}}
\newcommand{\beqs}{\begin{equation}}
\newcommand{\eeqs}{\end{equation}}
\newcommand{\beqa}{\begin{equation}\begin{aligned}}
\newcommand{\eeqa}{\end{aligned}\end{equation}}
\newcommand{\beqas}{\begin{equation}\begin{aligned}}
\newcommand{\eeqas}{\end{aligned}\end{equation}}
\newcommand{\brmk}{\begin{rmk}}
\newcommand{\ermk}{\end{rmk}}
\newcommand{\partref}[1]{\hbox{(\csname @roman\endcsname{\ref{#1}})}}
\newcommand{\half}{\frac{1}{2}}

\newcommand{\lie}{\mathcal{L}}

\newcommand*\tr{\mathop{\mathrm{tr}}\nolimits}

\newcommand*\dist{\mathop{\mathrm{dist}}\nolimits}

\newcommand*\arsinh{\mathop{\mathrm{arsinh}}\nolimits}

\newcommand*\ddt{\frac{\mathrm{d}}{\mathrm{d}t}}

\newcommand{\pt}{\partial_t}
\newcommand{\M}{\ensuremath{{\mathcal M}}_{-1}}

\newcommand{\abs}[1]{\vert#1\vert} 
 
\newcommand{\eps}{\varepsilon}
\newcommand{\na}{\nabla}
\newcommand{\Qu}{\mathcal{Q}}
\newcommand{\Hol}{{\mathcal{H}}} 

\newcommand{\ran}{\rangle}
\newcommand{\lan}{\langle}
\newcommand{\Col}{\mathcal{C}}
\newcommand{\thin}{\text{-thin}}
\newcommand{\thick}{\text{-thick}}
\newcommand{\Xd}{X_{\delta}}

\title{{\sc
Teichm\"uller harmonic map flow into nonpositively
curved targets
}
\\ 
}
\author{Melanie Rupflin and Peter M. Topping}
\date{\today}

\begin{document}
\maketitle
\begin{abstract}
The Teichm\"uller harmonic map flow deforms both a map from an oriented closed surface $M$ into an arbitrary closed Riemannian manifold, and a constant curvature metric on $M$, so as to reduce the energy of the map as quickly as possible \cite{RT}. The flow then tries to converge to a branched minimal immersion when it can \cite{RT, RTZ}. The only thing that can stop the flow is a finite-time degeneration of the metric on $M$ where one or more collars are pinched. In this paper we show that finite-time degeneration cannot happen in the case that the target has nonpositive sectional curvature, and indeed more generally
in the case that the target supports no bubbles.
In particular, when combined with  \cite{RT, RTZ, HRT}, 
this shows that the flow will decompose an arbitrary such map into a collection of branched minimal immersions.
\end{abstract}



\section{Introduction}

Given a smooth oriented closed surface $M:=M_\gamma$ of genus $\gamma\geq 2$ and a smooth closed Riemannian manifold 
$N=(N,G)$ of any dimension, we can imagine taking a gradient flow of the harmonic map energy 

$$E(u,g):=\half\int_M |du|^2_g dv_g,$$
simultaneously for both $u:M\to N$ a map and $g$ a hyperbolic (constant Gauss curvature $-1$) metric on $M$.
More precisely, given a fixed parameter $\eta>0$,
the Teichm\"uller harmonic map flow, introduced in \cite{RT},
is the flow defined by 
\begin{equation}
\label{flow}
\pl{u}{t}=\tau_g(u);\qquad \pl{g}{t}=\frac{\eta^2}{4} Re(P_g(\Phi(u,g))),
\end{equation}
where $\tau_g(u)$ represents the tension field of $u$ (i.e.
$\tr \nabla du$), $P_g$ represents the $L^2$-orthogonal projection from the space of quadratic differentials on $(M,g)$ onto the space of \emph{holomorphic} quadratic differentials, and
$\Phi(u,g)$ represents the Hopf differential -- see \cite{RT} for further information
and a description of the genus $\ga\leq 1$ cases.
The flow decreases the energy according to
\beq
\label{energy-identity}
\frac{dE}{dt}=-\int_M 
\left[|\tau_g(u)|^2+\left(\frac{\eta}{4}\right)^2 |Re(P_g(\Phi(u,g)))|^2\right]dv_g
.
\eeq

Given any initial data $(u_0,g_0)\in H^1(M,N)\times \M$, 
with $\M$ the set of smooth hyperbolic metrics on $M$, we know \cite{rupflin_existence} that a (weak) solution of \eqref{flow} exists on a maximal interval $[0,T)$, 
smooth except possibly at finitely many times,
and that $T<\infty$ only if the flow of metrics degenerates in moduli space
as $t\nearrow T$, that is if the length $\ell(g(t))$ of the shortest closed geodesic in $(M,g(t))$ converges to zero as $t\nearrow T$.
In the case that $T=\infty$, a description of the asymptotics of the flow was given in \cite{RTZ, HRT} 
(following on from \cite{RT}). Loosely speaking, it was shown that the surface $(M,g(t))$ can degenerate into finitely many lower genus surfaces, with the map $u(t)$ subconverging (modulo bubbling) to branched minimal immersions (or constant maps) on each of these components.

That theory immediately begs the question of whether the flow exists for all time (and thus enjoys this asymptotic convergence to minimal surfaces) or whether on the contrary, 
$\ell(g(t))$ can decay to zero in finite time, in which case a `collar' in the surface $(M,g(t))$ must pinch in finite time (see e.g. \cite{RT}).

In this paper we show that in the case that the target $(N,G)$ has nonpositive curvature, the Teichm\"uller harmonic map flow is very well behaved, with a \emph{smooth} solution existing for all time, given arbitrary initial data, and no bubbling occurring at infinite time. 
In fact, we prove this under the hypothesis that there does not exist any nonconstant harmonic map from $S^2$ to $(N,G)$, i.e. $(N,G)$ \emph{does not support any bubbles}, which is a more general result as we recall in Section \ref{sect_ruling}.
The flow then directly decomposes an arbitrary map into a collection of branched minimal immersions.

The theory of the classical harmonic map flow originated in the seminal paper of Eells and Sampson \cite{ES} in which the hypothesis of nonpositive curvature was also present. The essential idea in the classical case is that this hypothesis gives an upper bound on the energy density that is uniform in space and time. That is no longer true in our situation.
The main challenge in our work is to prevent the degeneration of collars (which makes no sense in the classical case) and the techniques we develop here are far removed from \cite{ES}. Our main result could be stated as:

\begin{thm}
\label{mainthm}
Suppose $M$, $(N,G)$ and $\M$ are as above, with $(N,G)$ having nonpositive sectional curvature, or more generally not supporting any bubbles.
Given any initial data $(u_0,g_0)\in C^\infty(M,N)\times \M$, 
there exists a smooth solution $(u(t),g(t))$ to \eqref{flow}, for $t\in [0,\infty)$.
\end{thm}

The proof of Theorem \ref{mainthm} will be somewhat involved, but at the coarsest level, it will turn out that the rate of collapse of a collar will be controlled by a weighted energy
$$\hat I=\half\int_M\frac{|du|^2(x)}{[\inj_g(x)]^2}dv_g(x),$$
(where $\inj_g(x)$ is the injectivity radius of $(M,g)$ at $x$) and a key part of this work will revolve around obtaining an upper bound for $\hat I$ over finite time intervals. Of course, the 
objective is to prevent energy from gathering on thin collars
where the injectivity radius is small. Our argument to deal with this involves an analysis that is reminiscent of the theory of neck analysis for almost harmonic maps \cite{QingTian, LinWang, toppingAnnals}, except our estimates must deal with the case where the energy on the collar is \emph{not} small, and where the tension field can be \emph{large}. We are also unable to use Hopf differential estimates to relate angular energy with radial energy on the collar, forcing us to deviate substantially from existing techniques.

Coupling our result with the asymptotic description of the flow from \cite{RTZ}, and using the 
hypothesis for the target once more, we will prove:

\begin{thm}
\label{thm:asymptotics}
In the situation of Theorem \ref{mainthm},
there exist a sequence of times $t_n\to \infty$, an integer $0\leq k\leq 3(\gamma-1)$ and a
hyperbolic punctured surface $(\Si,h,c)$ with $2k$ punctures (i.e. a closed Riemann surface $(\hat\Si,\hat c)$, possibly disconnected, that has been punctured $2k$ times and then equipped with a compatible complete hyperbolic metric $h$) such that the following holds.
\begin{enumerate}
\item The surfaces $(M,g(t_n),c(t_n))$ converge to the surface $(\Si,h,c)$ by collapsing $k$ simple closed geodesics $\sigma^{j}_{n}$
in the sense of Proposition \ref{Mumford}; in particular there is a sequence of diffeomorphisms $ f_n:\Si\to M\setminus \cup_{j=1}^k \sigma^{j}_{n}$
such that $$f_n^*g(t_n)\to h \text{ and } f_n^*c(t_n)\to c \text{ smoothly locally, }$$
where $c(t)$ denotes the complex structure of $(M,g(t))$.
 \item The maps $f_n^*u(t_n):=u(t_n)\circ f_n$ converge to a limit $u_\infty$ strongly in $W_{loc}^{2,2}(\Si)$.
\item 
The limit $u_\infty:\Si\to N$ extends to a smooth branched minimal immersion (or constant map) on each component of the compactification $(\hat\Si,\hat c)$ of $(\Sigma,c)$ obtained by filling in each of the $2k$ punctures.
\end{enumerate}
\end{thm}

\begin{rmk}
We do not claim that the image of $u_\infty$ must be connected. Indeed, as we show in \cite{HRT} with T. Huxol, the images of collapsing collars in $(M,g(t_n))$ can be mapped close to  nontrivial curves connecting the individual components of the image of $u_\infty$.
On the other hand, there can be no loss of energy on degenerating collars, see \cite{HRT}.
\end{rmk}

\begin{rmk}
In the light of the results above, it is interesting to compare the Teichm\"uller harmonic map flow with the mean curvature flow, which is also designed to flow to minimal maps. That flow can also be viewed as a flow of a pair $(u,g)$, where $u$ is an immersion that again satisfies the harmonic map flow, but the metric $g$ is set to equal the induced metric $u^*G$ at each moment, forcing it always to be conformal. Our main theorem \ref{mainthm} says that by imposing a curvature condition on the target, all singularities for the Teichm\"uller harmonic map flow can be eradicated at finite time. One cannot hope for a similar result for mean curvature flow.
\end{rmk}

Although it does not require the main innovations of this paper, when the target is negatively curved \emph{and}  the initial map $u_0$ is incompressible, a particularly clean conclusion follows using the work in \cite{RT}, including \cite[Remark 3.4]{RT} (see also Schoen-Yau \cite{SY} and Sacks-Uhlenbeck \cite{SU_TAMS}).

\begin{cor}
Suppose $(M,g_0)$ is an oriented closed hyperbolic surface and $(N,G)$ is a closed manifold with nonpositive sectional curvature. Suppose moreover that $u_0:M\to N$ is any smooth incompressible map. Then there exists 
a hyperbolic metric $\bar g$ on $M$ and a smooth branched minimal immersion $\bar u:(M,\bar g)\to (N,G)$ homotopic to $u_0$.

Indeed, there exist a global smooth solution $(u,g)$ of \eqref{flow}
for all $t\geq 0$, with $(u_0,g_0)$ as initial data, together with
sequences of times $t_n\to\infty$ and diffeomorphisms $f_n: M\to M$ isotopic to the identity such that 
\begin{enumerate}
\item
$f_n^*[g(t_n)]\to \bar g$ smoothly, and 
\item
$u(t_n)\circ f_n \to \bar u$ in $W^{2,2}(M,N)$, and in particular, in $C^0(M,N)$.
\end{enumerate}
\end{cor}

\begin{rmk}
As we will discuss elsewhere, in contrast to the theory of harmonic maps (Hartman \cite{hartman}) there can exist multiple branched minimal immersions within the same homotopy class of maps (even with disjoint image) even when the curvature of the target is strictly negative.
In this case, the corresponding domain metrics must represent different points in Teichm\"uller space.
\end{rmk}

\begin{rmk}
Although we state our results for closed target manifolds, the proof extends to somewhat more general situations that are important for applications we have in mind. For example, if $N$ is noncompact but supports a proper convex function (which will imply  the no-bubbles hypothesis \cite{gordon}) then the theory extends, with the image of the flow remaining within a compact region of $N$. Moreover, the theory extends to the case that $u$ is a section of certain twisted bundles, cf. Donaldson \cite{donaldson}.
\end{rmk}

The assumption of nonexistence of bubbles 
will be used in two main ways in the proof of Theorem \ref{mainthm}. Using well-understood principles (e.g. \cite{SU, struweCMH, DT} etc.) it prevents the energy of an 
almost-harmonic map from being too concentrated in isolated regions, which in turn allows us to make dramatically improved estimates in our collar analysis and proves that bubbling singularities cannot occur in the flow.
An unconventional feature of our work is that the hypothesis of nonexistence of bubbles will also allow us ultimately to control the weighted energy $\hat I$ mentioned earlier.
One can show that the rate of change of $\hat I$ can be controlled in terms of $\hat I^2$, but this is not enough to prevent finite time blow-up. Instead, we manage to control the evolution of $\hat I$ in terms of
 the product of $\hat I$ with only the 'angular' part of the weighted energy and this latter quantity can be controlled effectively under the hypothesis that the target admits no bubbles.

A substantial part of the paper is devoted not so much to the flow \eqref{flow} but to the study of more general curves of hyperbolic metrics, and (as a result) to the study of holomorphic quadratic differentials and the corresponding projection operator $P_g$. Some of this may be of independent interest; for example whereas $P_g$ is bounded from $L^2$ to $L^2$ by definition, Proposition \ref{lemma:projection}, asserts the useful fact that $P_g$ is bounded from $L^1$ to $L^1$ independent of how degenerate the underlying metric $g$ is.

This paper is organised as follows. In Section \ref{sect_ruling} we prove a formula for $\frac{d\ell}{dt}$ as we move $g$ in a general direction $Re(P_g(\Psi))$ (Lemmata \ref{lengthevol} and \ref{thmf_length_evol}) and assemble the proof of our main Theorem \ref{mainthm} based on growth estimates for the weighted energy (Lemma \ref{lemma:Icontrol}).
In Section \ref{angularsect} we control the angular energy on collars, the main result being Lemma \ref{ang_en_est}.
In Section \ref{length:evol:sect} we develop our understanding of holomorphic quadratic differentials in order to prove the formula for $\frac{d\ell}{dt}$ given in Lemma \ref{lengthevol}.
In Section \ref{sec:I} we establish that the full weighted energy can grow at most exponentially fast (Lemma \ref{lemma:Icontrol}), the key result being Lemma \ref{lemma:I-zeppelin}, and hence will remain bounded over finite time intervals, as required in Section \ref{sect_ruling}.

\emph{Acknowledgements:} 
We thank Scott Wolpert, Sumio Yamada and Mike Wolf for discussions concerning the existing theory of Weil-Petersson geometry.
The second author was supported by 
EPSRC grant number EP/K00865X/1.

\section{Ruling out collar degeneration}
\label{sect_ruling}

In this section, we assemble the proof of Theorem \ref{mainthm}, giving a global smooth solution of our flow. 
The first  point to verify is that, as is well-known and claimed in the introduction, the nonexistence of bubbles in the target is a more general hypothesis than the nonpositivity of its curvature.

\begin{lemma}
\label{nobubbles}
If $(N,G)$ is a complete Riemannian manifold of nonpositive sectional curvature, then every harmonic map $S^2\to (N,G)$ is a constant map.
\end{lemma}
\begin{proof}
By lifting to the universal cover, we may assume that $(N,G)$ is simply connected. The squared distance function $d^2(x_0,\cdot)$ to any fixed point $x_0\in N$ is a strictly convex function because of the nonpositive curvature \cite{BO}, and any harmonic map from any closed Riemannian manifold into any Riemannian manifold supporting a convex function is necessarily constant because the composition of the map and the convex function must be subharmonic \cite{gordon}.
\end{proof}

The starting point of the proof of Theorem \ref{mainthm}
is the weak solution constructed in \cite{rupflin_existence}, that exists until such a time $T$ that the length of the shortest closed geodesic converges to zero. 
As is well understood from \cite{struweCMH, rupflin_existence}, the only obstruction to a weak solution being smooth is the development of singularities at which 
one can perform a standard rescaling procedure to extract a bubble, i.e. a nonconstant harmonic map from $S^2$ to the target $(N,G)$
(cf. \cite{RTZ} and \cite{struweCMH}). However, no such bubble  exists in our situation by hypothesis.

Therefore, in the context of Theorem \ref{mainthm}, it remains
to show that the length of the shortest closed geodesic has a positive lower bound over arbitrary finite time intervals.

\subsection{Basics of Teichm\"uller theory}

Fix an oriented closed surface $M$ of genus $\ga\geq 2$, and consider the space $\m_{-1}$ of metrics $g$ on $M$ of constant Gauss curvature $-1$.
It is well understood (see for example \cite{RTZ} and Lemma \ref{lemma:collar}) that $(M,g)$ decomposes into a \emph{thick part} consisting of all points at which the injectivity radius is at least $\arsinh(1)$, and a finite collection of disjoint \emph{collar regions} $\Col_j$ for $1\leq j\leq k$.
Each collar region has at its centre a simple closed geodesic 
$\si_j$ of length $\ell_j$. (See Lemma \ref{lemma:collar}.)

Tangent vectors in $\m_{-1}$ at $g$ can be decomposed into the sum of a Lie derivative term $\lie_X g$ (corresponding to a change in parametrisation of the same metric) and a term of the form 
$Re(\Th)$, where $\Th$ lies in the $(3\ga-3)$-dimensional complex vector space $\Hol(M,g)$ of holomorphic quadratic differentials
(see e.g. \cite{tromba}).
Ultimately this allows us to view a smooth path in Teichm\"uller space as a 
smooth family $g(t)$ of metrics in $\M$ for which 
$\pl{g}{t}=Re(\Th(t))$ for some smooth family 
$\Th(t)\in\Hol(M,g(t))$, and it is these \emph{horizontal curves} that we study below.
In the next section we will need to understand how the lengths $\ell_j$ of the geodesics $\si_j$ evolve as $g(t)$ evolves in this way.

\subsection{Proof of the main theorem ruling out collar degeneration}

Our basic set-up for Theorem \ref{mainthm} is that we have a one-parameter family of metrics $g(t)$ evolving under the equation
$\pl{g}{t}= Re(P_g(\Psi(t)))$, with 
$\Psi(t)=\frac{\eta^2}{4} \Phi(u,g)$, and we need to control the evolution of the length of the shortest closed geodesic in order to prevent it from decreasing to zero in finite time, which would correspond to the degeneration of a collar. Our basic result in this direction is the following lemma, which may be of independent interest. 
We write $\Qu_{L^2}(M,g)$ for the infinite dimensional space 
of measurable quadratic differentials on $(M,g)$ with finite $L^2$ norm.
To fix normalisations, take an arbitrary local complex coordinate $z=x+iy$ and write $g=\rho^2(dx^2+dy^2)$, so 
$\Psi=\psi (dx+idy)^2$ for some locally defined complex valued $L^2$ function $\psi$. We are then normalising so that
$$\norm{\Phi}_{L^2(M,g)}^2=\int_M\abs{\psi}^2\abs{dz^2}^2 dv_g=4\int_M \rho^{-2}\abs{\psi}^2 dx\wedge dy.$$
In the lemma, we consider only the complex coordinate $z:=s+i\th$, where $(s,\th)$ are cylindrical coordinates on the collar  
(see Lemma \ref{lemma:collar}) and the corresponding $dz^2=(ds+id\th)^2$.


\begin{lemma}
\label{lengthevol}
Given an oriented closed surface $M$ of genus $\ga\geq 2$, there exists 
$C<\infty$ depending only on $\ga$ such that the following is true.
Suppose $g(t)$ is a smooth one-parameter family of metrics in $\M$ for $t$ in a neighbourhood of $0$ such that 
at $t=0$, we have 
$$\pl{g}{t}=Re(P_g(\Psi))\quad\text{ for some }\quad\Psi\in \Qu_{L^2}(M,g(0)),$$ 
and we have a collar $\Col$ in $(M,g(0))$ around a simple closed geodesic of length $\ell< 
2\arsinh(1)$.
Then 
$$\frac{d\ell}{dt}\sim -\frac{\ell^2}{16\pi^3}Re\langle\Psi,dz^2\rangle_{L^2(\Col,g)},$$
at $t=0$, in the sense that 
$$\bigg|\frac{d\ell}{dt}+\frac{\ell^2}{16\pi^3}Re\langle\Psi,dz^2\rangle_{L^2(\Col,g)}\bigg|
\leq C\ell^2\|\Psi\|_{L^1(M,g)}.$$
\end{lemma}

Note that the content of this lemma revolves around the fact that $\Psi$ need not be holomorphic but
rather can be any element of 
$\Qu_{L^2}(M,g)$. 
One could get a feel for Lemma  \ref{lengthevol} by using it to reprove the incompleteness of Teichm\"uller space, 
in which case we take $\Psi$ to be $dz^2$ on the collar and zero elsewhere, and the error term is a factor of $\ell^2$ smaller than the leading term for $\frac{d\ell}{dt}$ (see Section \ref{TeichIncomp}). 
Formulae for the first and even second derivatives of $\ell$ of a quite different flavour to ours
can be found in \cite{gardiner, wolf} and the references therein.

Lemma \ref{lengthevol} will be proved in 
Section \ref{length:evol:sect}, based on an analysis of the space of holomorphic quadratic differentials. We use it now in the special case of the Teichm\"uller harmonic map flow, to prove:

\begin{lemma}
\label{thmf_length_evol}
Let $(u,g)$ be a smooth
solution of the Teichm\"uller harmonic map flow \eqref{flow}
defined on a surface of genus at least two and on a time interval $[0,T)$.
Given a collar $\Col$ in $(M,g)$ at time $t$, with central geodesic $\si$ of length $\ell<2\arsinh(1)$ 
we have
\beq  \label{est:ddl} \bigg|\frac{d}{dt}\log(\ell)+\frac{\eta^2}{16\pi^3}\cdot \ell \int_\Col (\abs{u_s}^2-\abs{u_\theta}^2)\rho^{-2} dsd\theta\bigg|
\leq C \ell \eta^2  E_0,\eeq
where $E_0$ is an upper bound for the energy and $C<\infty$ depends only on the genus $\ga$. 
In particular, the evolution of $\ell$ is controlled in terms of a weighted energy
$$I:=\int_\Col e(u,g)\rho^{-2}dv_g,$$
where $e(u,g)=\half(|u_s|^2+|u_\th|^2)\rho^{-2}$ is the energy density and  $\rho$ is the conformal factor defined in Lemma \ref{lemma:collar}, 
in the sense that 
$$\bigg|\frac{d}{dt}\log\ell\bigg|\leq C\ell \big[I+E_0\big],$$
with $C$ depending only on $\ga$ and the coupling constant $\eta$.
\end{lemma}

\begin{proof}
Under Teichm\"uller harmonic map flow, the metric $g$ evolves
according to
$$\pl{g}{t}=\frac{\eta^2}{4}Re(P_g(\Phi(u,g))),$$
where $\Phi(u,g)=(|u_s|^2-|u_\th|^2-2i\langle u_s,u_\th\rangle)dz^2$ on the collar.
Therefore in the language of Lemma \ref{lengthevol}, we have
$$Re\langle\Psi,dz^2\rangle_{L^2(\Col,g)}=\frac{\eta^2}{4}
\int_\Col (|u_s|^2-|u_\th|^2)|dz^2|^2\rho^2dsd\th
=\eta^2\int_\Col (|u_s|^2-|u_\th|^2)\rho^{-2}dsd\th,$$
by \eqref{sizes_on_collars}.
Note also that 
$\|\Phi\|_{L^1}\leq 4 E_0$, so 
$$\|\Psi\|_{L^1(M,g)}\leq \eta^2 E_0.$$
Therefore Lemma \ref{lengthevol} implies that
\beqs
-\frac{d}{dt}\log\ell\sim \frac{\eta^2}{16\pi^3}\cdot \ell \int_\Col (\abs{u_s}^2-\abs{u_\theta}^2)\rho^{-2} dsd\theta \eeqs
up to an error that is bounded by \eqref{est:ddl} for $C$  depending only on the genus $\ga$.
\end{proof}

We see from Lemma \ref{thmf_length_evol} that we could deduce that $\ell(g(t))$ does not decrease to zero in finite time if we could prove that $\ell I$ remains bounded over 
arbitrary compact time intervals. In fact, we will prove the stronger statement that $I$ itself remains bounded. 
This will imply that the smooth solution of the Teichm\"uller harmonic map flow discussed above
must exist for all time,
which will complete the proof of Theorem \ref{mainthm}.
\newcommand{\I}{{\mathcal{I}}}

\begin{lemma}
\label{lemma:Icontrol}
Suppose $M$, $(N,G)$ and $\M$ are as above, with $(N,G)$ 
supporting no bubbles. Then for any $E_0>0$ there exists a constant
$C<\infty$ such that the following holds true. 
Let $(u,g)$ be any solution of the Teichm\"uller harmonic map flow \eqref{flow}, defined on an interval $[0,T)$, with initial energy $E(u(0),g(0))\leq E_0$. 
Then for any time $t\in [0,T)$ such that $(M,g(t))$ contains a collar $\Col$ with central geodesic of length $\ell< 2\arsinh(1)$, 
we can estimate the weighted energy defined above by
$$I:=\int_\Col e(u,g)\rho^{-2}dv_g\leq Ce^{Ct}
(1+(\inj_{g(0)}M)^{-2}),$$
where $\inj_g M:=\inf_{x\in M}\inj_g(x)$.
\end{lemma}

We remark that at each point of such a collar the injectivity radius $\inj_g(p)$ and the conformal factor $\rho(p)$ 
are of comparable size, see \eqref{est:rho-inj-appendix} and \eqref{est:rho-by-inj}, 
so bounding $I$ on each such collar is indeed equivalent to bounding the global weighted energy $\hat I$ briefly mentioned before. 

We will prove Lemma \ref{lemma:Icontrol} in Section \ref{sec:I}, but before that, in Section \ref{angularsect} we will derive estimates on a weighted \emph{angular} energy alone --  see Lemma \ref{ang_en_est}. This is ironic given that a careful reading
of \eqref{est:ddl} shows that large angular energy appears to help \emph{prevent} degeneration of the neck. 
However, the proof of Lemma \ref{lemma:Icontrol} will bootstrap angular energy estimates to full energy estimates.
One might expect this to follow using the Hopf differential $\Phi(u,g)$, which in some sense measures the difference between angular energy and full energy. 
However, although estimates on the Hopf differential
do follow from the flow equations, the obvious ones are not strong enough for our purposes, and so instead we use a dynamic argument
in the proof of Lemma \ref{lemma:Icontrol}. 
That part is based on a precise understanding of the evolution of the metric, and in particular the weight $\rho^{-2}$, on a collar,
as well as the smallness of the angular energy, which in turn crucially uses the assumption of nonexistence of bubbles.

\subsection{Asymptotics}
\label{section:asympt}

In this section we make the final observations required to prove
Theorem \ref{thm:asymptotics}. 
Theorem \ref{mainthm} already gives global smooth existence for the Teichm\"uller harmonic map flow, and \cite[Theorem 1.1]{RTZ} and \cite[Theorem 1.4]{RT} describe the decomposition of the
flow into branched minimal immersions with the required level of convergence except at a finite set $S$ of points in $M$ at which bubbling occurs; but $S$ must be empty by the no-bubbles hypothesis on $(N,G)$.

Note that while 
\cite[Theorem 1.1]{RTZ} and \cite[Theorem 1.4]{RT} only state convergence of $f_n^*u(t_n)$ in each $W^{1,p}_{loc}$, $p<\infty$, in the proof the sequences are chosen so that 
$\norm{\tau_g(u)(t_n)}_{L^2}\to 0$, which, when combined with the convergence of the metrics and the local $W^{1,p}$ convergence away from $S$, implies
that $\Delta_{h}(f_n^*u(t_n)-u_{\infty})$ converges to zero (strongly) in $L^2_{loc}(\Si\setminus S)$, thus giving the desired local $H^2$ convergence.

\section{Controlling the angular energy}
\label{angularsect}

Our goal in this section is to control a weighted angular energy that is similar to the weighted energy $I$ from Lemma \ref{thmf_length_evol}, but only considers $\th$ derivatives.
More precisely, we prove the following key lemma.

\begin{lemma}
\label{ang_en_est}
For any $E_0<\infty$ and closed Riemannian manifold $N$ not supporting any bubbles, 
there exists $C<\infty$ 
such that for any $\ell\in (0,2\arsinh(1))$ and map $u:(\Col(\ell),g)\to N$ from a hyperbolic collar, with energy $E(u)\leq E_0$, the angular energy is
controlled according to
$$I^{(\th)}:=
\int_{\Col(\ell)}
\rho^{-2}|u_\th|^2d\th ds
\leq C(1+\|\tau_g(u)\|_{L^2(\Col(\ell),g)}^2).$$
\end{lemma}

Of course, without the weighting coefficient $\rho^{-2}$ 
(with $\rho$ from Lemma \ref{lemma:collar}) the left-hand side would be the normal angular energy, and would thus be bounded. When additionally the tension can be controlled, we will show that the angular energy decays exponentially along the collar, and this decay dominates the growth of $\rho^{-2}$ towards the centre of the collar.

We will give the proof of Lemma \ref{ang_en_est} towards the end of Section \ref{angensect} once we have developed some preliminary theory.

\subsection{Controlling the concentration of energy}

\newcommand{\Cyl}{{\mathscr{C}}}

In this section we elaborate on the well-known principles that regions of concentrated energy in almost harmonic maps (i.e. maps with small tension field) can be blown up to yield bubbles, and that concentrated energy poses the only obstruction to getting higher order estimates.
(Recall from the introduction that a bubble is a nonconstant harmonic map from $S^2$ to $N$.)
See (for example) Corollary \ref{cor6} for a consequence of these principles.

Given $s\in\R$ and $\Lambda\in (0,\infty)$, we define the cylinder
$\Cyl_\Lambda(s)$ to be $(s-\Lambda,s+\Lambda)\times S^1$. 
By default, this will be equipped with the standard cylindrical metric $ds^2+d\th^2$, in which case we drop references to the metric, for example abbreviating the tension by $\tau(u)$ or simply $\tau$. The metric is only made explicit in the case that we equip the cylinder with 
(part of) a hyperbolic collar metric $g$.

\begin{lemma}
\label{small_energy}
Given a closed Riemannian manifold $N$ not supporting any bubbles, and constants
$\ep_1>0$ and $E_0<\infty$, there exist $\tilde \Lambda\in(1,\infty)$ and $K<\infty$ such that the following holds.
Given any smooth map $u:\Cyl_{\tilde \Lambda}(0)\to N$ with total energy
$E(u;\Cyl_{\tilde\Lambda}(0))\leq E_0$, 
we have that 
$$E(u;B_{r_0}(p))<\ep_1\qquad\text{for all }
p\in \{0\}\times S^1$$
for all $r_0\leq(1+K\|\tau(u)\|_{L^2(\Cyl_{\tilde \La}(0))})^{-1}$.
\end{lemma}

We stress that all quantities in the lemma above are computed with respect to the flat metric $ds^2+d\th^2$, despite the fact that we will often apply it to cylinders on which we have 
a hyperbolic metric.

\begin{proof}
We proceed by contradiction: if the lemma were false, then there would exist $\ep_1>0$, $E_0\in (0,\infty)$ and 
(after rotating the cylinder)
a point $p\in\{0\}\times S^1$, and a sequence of smooth maps $u_i:\Cyl_i(0)\to N$ with $E(u;\Cyl_i(0))\leq E_0$, 
such that 
\begin{equation}
\label{energy_lower_bound}
E(u_i;B_{r_i}(p))\geq\ep_1
\end{equation}
for $r_i=(1+i\|\tau(u_i)\|_{L^2(\Cyl_i(0))})^{-1}\in (0,1]$.
First we consider the case that after passing to a subsequence
we can arrange that $\|\tau(u_i)\|_{L^2(\Cyl_i(0))}\to 0$.
In this case, we can perform a standard bubbling analysis
to the sequence $u_i$ in order to extract a bubble. More precisely, we can pass to a subsequence and extract a local weak $W^{1,2}$-limit $u_\infty:(-\infty,\infty)\times S^1\to N$, which is harmonic, and by the Sacks-Uhlenbeck removable singularity theorem \cite{SU} we can add two points at infinity in the cylinder 
$(-\infty,\infty)\times S^1$ and extend $u_\infty$ to a harmonic map $S^2\mapsto N$.
Since there exist no bubbles by hypothesis, $u_\infty$ must be a constant map. 
But then the bubbling theory, combined with \eqref{energy_lower_bound} tells us that we can blow up the maps $u_i$ about an appropriate sequence of points $p_i\in B_{r_i}(p)$ and extract a nonconstant harmonic limit $\tilde u_\infty:\R^2\to N$ in $W^{2,2}_{loc}(\R^2,N)$, which can be extended (by adding a point at infinity) to a nonconstant harmonic map $S^2\mapsto N$, i.e. a bubble, which by hypothesis cannot exist. We have arrived at a contradiction in the case that the tension decays to zero (for a subsequence).

The remaining case is that $\|\tau(u_i)\|_{L^2(\Cyl_i(0))}$ is bounded below by some positive constant, uniformly in $i$, which forces $r_i\to 0$. In this case, we blow up each map $u_i$ by a factor $r_i$, and end up with a sequence of maps $\tilde u_i$ on fatter and fatter cylinders, so that $E(\tilde u_i,B_1)\geq \ep_1$ for each $i$, and so that 
$$\|\tau(\tilde u_i)\|_{L^2}=r_i
\|\tau(u_i)\|_{L^2(\Cyl_i(0))}
= \frac{1-r_i}{i}\to 0$$
as $i\to\infty$.
A similar bubbling argument to before allows us to extract a bubble, giving a contradiction in this case too.
\end{proof}

We will combine the lemma above with the following standard regularity estimate in which $B_r$ represents the disc of radius $r>0$ in the flat plane.

\begin{lemma}[{cf. \cite[Lemma 3.3]{RT}}]
\label{basicW22W14}
Given a closed target $N$, there exists $\ep_0>0$ and 
$C<\infty$ such that 
for any $r>0$, and smooth map $u:B_r\to N$ with $E(u;B_r)\leq \ep_0$, we have
$$\int_{B_{r/2}}
|\grad^2 u|^2+|\grad u|^4
\leq C\left(\frac{E(u;B_r)}{r^2}+\|\tau(u)\|_{L^2(B_r)}^2\right).$$
\end{lemma}
Note that the estimate for $\|\grad^2 u\|_{L^2}^2$ is the standard one. The estimate for $\|\grad u\|_{L^4}^4$ follows by applying Sobolev to $|\grad u|^2$ to yield
$$\int_{B_{r/2}}|\grad u|^4\leq
C E(u;B_{r/2})\left[
\int_{B_{r/2}}|\grad^2 u|^2+r^{-2}E(u;B_{r/2})
\right],
$$
and bounding the first $E(u;B_{r/2})$ by $\ep_0$.

The combination of Lemma \ref{basicW22W14} and Lemma \ref{small_energy} will yield:

\begin{cor}
\label{cor5}
For any $E_0<\infty$, and closed Riemannian manifold $N$ not supporting any bubbles, there exist $\La\in(2,\infty)$ and $C<\infty$ such that the following holds.
Given any smooth map $u:\Cyl_\La(0)\to N$ with total energy
$E(u;\Cyl_\La(0))\leq E_0$, we have that 
$$\int_{\Cyl_1(0)}|\grad^2 u|^2+|\grad u|^4
\leq C\left(E(u;\Cyl_2(0))+\|\tau(u)\|_{L^2(\Cyl_\La(0))}^2\right).$$
\end{cor}

\begin{proof}
For our given target $N$, let $\ep_1$ be the $\ep_0$ from Lemma \ref{basicW22W14} and feed $\ep_1$ into Lemma \ref{small_energy} together with our $E_0$, to obtain in particular 
the constants $\tilde \La$ and $K$.
We set $\La:=\tilde \La+1$. For $u$ as in the corollary, we then define 
$r_0=(1+K\|\tau(u)\|_{L^2(\Cyl_{\La}(0))})^{-1}$ so that we will be able to apply Lemma \ref{small_energy} on cylinders $\Cyl_{\tilde \La}(s)$ for $s\in [-1,1]$.
We now cover $\Cyl_1(0)$ by balls $B_{r_0/2}(p_i)$ of radius $r_0/2$, with  $p_i\in \Cyl_1(0)$; moreover, we can achieve this with no more than $Cr_0^{-2}$ balls, and so that each point in $\Cyl_2(0)$ is covered no more than $C$ times by the balls $B_{r_0}(p_i)$, for universal $C$.
Adding the estimates from Lemma \ref{basicW22W14} for each of these balls, we obtain
\beqas
\int_{\Cyl_1(0)}
|\grad^2 u|^2+|\grad u|^4
&\leq \sum_i\int_{B_{r_0/2}(p_i)}
|\grad^2 u|^2+|\grad u|^4
\\
&\leq
C\sum_i\left(\frac{E(u;B_{r_0}(p_i))}{r_0^2}+\|\tau(u)\|_{L^2(B_{r_0}(p_i))}^2\right)\\
&\leq
C\left(\frac{E(u;\Cyl_2(0))}{r_0^2}+\|\tau(u)\|_{L^2(\Cyl_2(0))}^2\right).
\eeqas
Using the formula above for $r_0$, we deduce
$$
\int_{\Cyl_1(0)}
|\grad^2 u|^2+|\grad u|^4
\leq
CE(u;\Cyl_2(0))(1+K\|\tau(u)\|_{L^2(\Cyl_{\La}(0))})^2
+C\|\tau(u)\|_{L^2(\Cyl_2(0))}^2,
$$
and hence
\beqs
\int_{\Cyl_1(0)}
|\grad^2 u|^2+|\grad u|^4
\leq
C\left(E(u;\Cyl_2(0))+
\|\tau(u)\|_{L^2(\Cyl_{\La}(0))}^2\right)
\eeqs
as desired, since $C$ is allowed to depend on $E_0$.
\end{proof}

Although we will need the corollary above in the given form, we will also be able to simplify its conclusion using the following:

\begin{cor}
\label{cor6}
For any $E_0<\infty$, $\ep_2>0$, and closed Riemannian manifold $N$ not supporting any bubbles,
there exist $\La\in(3,\infty)$ and $C<\infty$ such that the following holds.
Given any smooth map $u:\Cyl_\La(0)\to N$ with total energy
$E(u;\Cyl_\La(0))\leq E_0$, we have that 
$$E(u;\Cyl_2(0))\leq \ep_2+C\|\tau(u)\|_{L^2(\Cyl_{\La}(0))}^2,$$
and 
$$\int_{\Cyl_1(0)}|\grad^2 u|^2+|\grad u|^4
\leq \ep_2+C\|\tau(u)\|_{L^2(\Cyl_\La(0))}^2.$$
\end{cor}

\begin{proof}
First, we pick $C_0<\infty$ large enough so that for any $r\in (0,1]$, the cylinder $\Cyl_2(0)$ can be covered by $C_0r^{-2}$ balls in $(-\infty,\infty)\times S^1$ of radii $r$, and with centres in $\Cyl_2(0)$.
We may then define $\ep_1:=\frac{\ep_2}{2C_0}$, and appeal to 
Lemma \ref{small_energy} to obtain $\tilde \La$ and $K$. We fix $\La:=\tilde \La+2$, and consider a map $u:\Cyl_\Lambda(0)\to N$.
By setting $r_0:=(1+K\|\tau(u)\|_{L^2(\Cyl_{\La}(0))})^{-1}$,
we are then able to cover $\Cyl_2(0)$ by $C_0{r_0}^{-2}$ balls $B_{r_0}(p_i)\subset (-\infty,\infty)\times S^1$ with $p_i\in\Cyl_2(0)$, and we may 
apply Lemma \ref{small_energy} for each $i$ to obtain bounds $E(u;B_{r_0}(p_i))<\ep_1$.
Summing these estimates yields
\beqs
\begin{aligned}
E(u;\Cyl_2(0))&\leq 
\sum_i E(u;B_{r_0}(p_i))<C_0{r_0}^{-2}\ep_1\\
&=C_0(1+K\|\tau(u)\|_{L^2(\Cyl_{\La}(0))})^2\frac{\ep_2}{2C_0}\\
&\leq \ep_2+C\|\tau(u)\|_{L^2(\Cyl_{\La}(0))}^2,
\end{aligned}
\eeqs
which is the first part of the corollary. 
By combining what we have proved with Corollary \ref{cor5}, allowing $ \La$ (and $C$) to increase and $\ep_2$ to decrease as necessary, we deduce the second part of the corollary.
\end{proof}

\subsection{Estimates on the angular energy}
\label{angensect}

In this section we prove Lemma \ref{ang_en_est}, telling us that the weighted angular energy is controlled in terms of the tension field.

As usual, we will be implicitly using the collar lemma \ref{lemma:collar} and its notation. In particular given a hyperbolic collar $(\Col(\ell),g)$, and writing $\Col=(-X(\ell),X(\ell))\times S^1$, we let 
$X_\de=X_\de(\ell)$ be the number given in \eqref{eq:Xde} for which the $\de\thin$ part of the collar is described (in collar coordinates) by $(-X_\de,X_\de)\times S^1$.

The main ingredient is the following result, which forces the angular energy on unit-length chunks of our long cylinder to decay exponentially as we move in from the ends of the cylinder, at least when the tension field is suitably small.

\begin{prop}\label{prop:angular-energy}
Let $N$ be a closed target that supports no bubbles. Then given any $E_0<\infty$ there exist numbers $C<\infty$, $\La<\infty$
and $\de\in(0,\arsinh(1))$ 
such that for any $\ell\in (0,2\arsinh(1))$, 
any
map $u$ from the hyperbolic collar $(\Col(\ell),g)$ to $N$ with total energy $E(u;\Col(\ell))\leq E_0$, and
any
$s_0\in (-\Xd(\ell),\Xd(\ell))$ we have
$\Cyl_\La(s_0)\subset \Col(\ell)$ and 
$$\int_{s_0-\frac12}^{s_0+\frac12}\int_{S^1}\abs{u_\theta}^2 d\theta ds\leq C\cdot e^{-(X(\ell)-\abs{s_0})} 
+C\cdot \int_{-X_\de(\ell)}^{X_\de(\ell)}e^{-\abs{s-s_0}}\rho^2(s)\cdot \norm{\tau_g(u)}_{L^2(\Cyl_\La(s),g)}^2 ds.$$
\end{prop}

In the situation that $u$ is a map with \emph{small} tension and small energy, there is some history of results that show exponential decay of energy along cylinders, starting with Hadamard's three circle theorem, and including \cite{QingTian, LinWang, toppingAnnals}. 
The key to such results is typically to derive a second order differential inequality for the angular energy over circles $\{s\}\times S^1$ and then apply the maximum principle.
Our present situation is more complicated because we don't make any \emph{a priori} assumptions of smallness of energy or tension, which prevents us from deriving angular energy estimates on such circles.
Instead, we derive a second order `delay' differential inequality for an angular energy averaged over short lengths of cylinder, defined by:

$$\Theta(s_0):=\int_{-X(\ell)}^{X(\ell)}\int_{S^1}\varphi^4(s-s_0)\abs{u_\theta(s,\theta)}^2 d\theta ds,$$ 
where $\varphi\in C^\infty_0((-1,1),[0,1])$ 
is a cut-off function
with $\varphi\equiv 1$ on $[-\frac12,\frac12]$,
and $|s_0|\leq X(\ell)-1$.

\begin{lemma}\label{lemma:diff-inequality}
Let $N$ be a closed target that supports no bubbles. Then given any $E_0<\infty$ there exist numbers $\de\in(0,\arsinh(1))$, $C_1\in[0,\infty)$ and 
$\La<\infty$
such that for any $\ell\in (0,2\arsinh(1))$,
any map $u:(\Col(\ell),g)\to N$ with total energy 
$E(u;\Col(\ell))\leq E_0$, and any $s\in (-\Xd(\ell),\Xd(\ell))$,
we have $\Cyl_\La(s)\subset \Col(\ell)$, and 
the differential inequality
\beq\label{est: diff-inequality}
\Theta''(s)-\frac32\Theta(s)+\frac18\bigg[\Theta(s-\frac12)+\Theta(s+\frac12)\bigg]\geq -C_1\rho^2(s)\norm{\tau_g(u)}_{L^2(\Cyl_\La(s),g)}^2
\eeq
is satisfied for the angular energy function $\Th$ associated with $u$.
\end{lemma}

Accepting this lemma for the moment, we can give the:

\begin{proof}[Proof of Proposition \ref{prop:angular-energy}]
For the given target $N$ and energy upper bound $E_0$, let $\de>0$, $C_1$ and $\Lambda$ be as in Lemma 
\ref{lemma:diff-inequality}. We fix arbitrary $\ell\in (0,2\arsinh(1))$, although if $\ell\geq 2\de$ we have
$(-X_\de,X_\de)=\emptyset$, so there is nothing to prove.
(Recall that we sometimes abbreviate $X_\de:=X_\de(\ell)$.)

We then consider a map 
$u:(\Col(\ell),g)\to N$ with total energy $E(u;\Col(\ell))\leq E_0$.
Let 
$$L(f)=f''(s)-\frac32 f(s)+\frac18\big[f(s+\frac12)+f(s-\frac12)\big]$$
be the operator describing the left-hand side of the differential inequality \eqref{est: diff-inequality} satisfied by the angular energy function $\Theta$.
We observe that while $L$ is not a classical differential operator, 
the usual comparison principle for ODE still applies. 
More precisely, let $f,\tilde f\in C^2([-\Xd,\Xd])$ be any two functions 
such that 
$$L(f)\leq L(\tilde f) \text{ on } (-\Xd+\frac12,\Xd-\frac12) \text{ with } f\geq \tilde f \text{ on } [-\Xd,-\Xd+\frac12]\cup [\Xd-\frac12, \Xd].$$
Then $f\geq \tilde f$ on all of  $[-\Xd,\Xd]$; indeed, if $f-\tilde f$ were to achieve a negative minimum at some $s_0\in (-\Xd+\frac12,\Xd-\frac12)$,
then
$$(f-\tilde f)''(s_0)\leq \frac32 \min (f-\tilde f)-\frac18\big[(f-\tilde f)(s_0+\frac12)+(f-\tilde f)(s_0-\frac12)\big]
\leq (\frac32-\frac14)\min (f-\tilde f)<0$$
would lead to a contradiction. 

In order to bound the angular energy function $\Th$, we compare it with solutions of a slightly modified equation, namely 
of
\beq \label{eq:ODE} 
f''-f=-C_1\cdot G \qquad\text{ on } [-\Xd,\Xd], \eeq
where $G(s):= \rho^2(s)\norm{\tau_g(u)}_{L^2(\Cyl_\La(s),g)}^2$ 
is the function on the right-hand side of \eqref{est: diff-inequality}, modulo the constant $-C_1$. 

Recall that any solution of \eqref{eq:ODE}
can be described by 

\beqs
f_{A,B}(s):=A\cdot e^{s-\Xd}+B\cdot e^{-s-\Xd}+\frac{C_1}2\int_{-\Xd}^{\Xd}e^{-\abs{s-q}}G(q)dq,\quad A,B\in\R.
\eeqs
Because $C_1\geq 0$, if
$A,B>0$ then the functions $f_{A,B}$ are positive and 
$$\frac18\left(f_{A,B}(s+\frac12)+f_{A,B}(s-\frac12)\right)\leq \frac{e^{\frac12}}4\cdot f_{A,B}(s)<\frac12f_{A,B}(s)$$
for any $s\in[-\Xd+\frac12,\Xd-\frac12]$. 
Consequently, $L(f_{A,B})\leq -C_1 G(s)\leq L(\Theta)$. 
Since we always have $\Theta(s)\leq 2E_0$ for any $s$, we may apply the comparison theorem with $A,B=2E_0e$, say, to conclude that
$$\Theta(s)\leq f_{A,B}(s)=2E_0\cdot (e^{s-\Xd+1}+e^{-s-\Xd+1})+\frac{C_1}2\int_{-\Xd}^{\Xd}e^{-\abs{s-q}}G(q)dq.$$
Since $X(\ell)-\Xd(\ell)\leq \frac{C}{\delta}$ is bounded independently of $\ell$, by Proposition \ref{X_Xdelta},
we thus obtain the claim of Proposition \ref{prop:angular-energy}.
\end{proof}

The proof above relied on Lemma \ref{lemma:diff-inequality}, claiming a second order differential inequality for the locally smoothed angular energy $\Th(s)$.

\begin{proof}[Proof of Lemma \ref{lemma:diff-inequality}]
Defining, for $s\in (-X(\ell),X(\ell))$,
$$\vth(s):=\int_{\{s\}\times S^1}|u_\th|^2,$$
where we drop the volume element $d\th$ for brevity,
we may compute
$$\vth'(s)=2\int_{\{s\}\times S^1}u_\th \cdot u_{s\th},$$
and
\beqs
\begin{aligned}
\vth''(s)&=2\int_{\{s\}\times S^1}|u_{s\th}|^2+u_\th \cdot u_{ss\th}\\
&= 2\int_{\{s\}\times S^1}|u_{s\th}|^2-u_{\th\th} \cdot u_{ss}.
\end{aligned}
\eeqs
Meanwhile, the tension $\tau$ of $u$ with respect to the flat cylindrical metric is given in terms of the second fundamental form $A(\cdot)$ of the target $(N,G)\embed \R^{N_0}$ by
$$\tau=u_{ss}+u_{\th\th}+A(u)(u_s,u_s)+A(u)(u_\th,u_\th)$$
and thus
$$\vth''(s)=
2\int_{\{s\}\times S^1}|u_{s\th}|^2+|u_{\th\th}|^2
-u_{\th\th} \cdot \tau
+u_{\th\th} \cdot \left[A(u)(u_s,u_s)+A(u)(u_\th,u_\th)\right].$$

We develop the penultimate term using integration by parts:
$$\left|\int u_{\th\th} \cdot [A(u)(u_s,u_s)]\right|
=
\left|\int u_{\th} \cdot [A(u)(u_s,u_s)]_\th\right|
\leq
C_N\int |u_\th|^2|u_s|^2+|u_{s\th}||u_s||u_\th|$$
and then apply Young's inequality to estimate
\beq
\label{vth_deriv}
\vth''(s)\geq 
(2-\frac14)\int_{\{s\}\times S^1}|u_{s\th}|^2+|u_{\th\th}|^2
-C\int_{\{s\}\times S^1}|\tau|^2
-C_{N}\int_{\{s\}\times S^1}|u_\th|^2|u_s|^2+|u_\th|^4.
\eeq

Let us now assume that $s_0$ is not too close to the ends of the cylinder, more precisely that $|s_0|\leq X(\ell)-1$, so 
that we can change variables 
($\tilde s=s-s_0$)
and write
$$\Theta(s_0):=\int\varphi^4(\tilde s)
\vth(\tilde s+s_0)d\tilde s.$$ 
Differentiating twice and using \eqref{vth_deriv}, we find that
\beqas
\Theta''(s_0)&=\int\varphi^4(\tilde s)
\vth''(\tilde s+s_0)d\tilde s
=\int\varphi^4(s-s_0)
\vth''(s)ds\\
&\geq
\int\varphi^4(s-s_0)
\left(
(2-\tfrac14)
(|u_{s\th}|^2+|u_{\th\th}|^2)
-C
|\tau|^2
-C_{N}
\left(|u_\th|^2|u_s|^2+|u_\th|^4\right)
\right)d\th ds.
\eeqas
Considering now $s_0$ to be fixed, we set $\tilde\vph(s):=\vph(s-s_0)$. Allowing $C$ to depend on $N$, we can then write
\beqa
\label{Th_deriv}
\Theta''(s_0)
&\geq
\int\tilde\varphi^4
\left(
(2-\frac14)(|u_{s\th}|^2+|u_{\th\th}|^2)
-C\left(|\tau|^2
+|u_\th|^2|u_s|^2+|u_\th|^4\right)
\right),
\eeqa
where we now start dropping the volume element $dsd\th$.

In order to control the final term, we
use the Sobolev inequality over 
the entire collar to find that
\beqa
\int\tilde\vph^4|u_\th|^4 &=\|\tilde\vph^2|u_\th|^2\|_{L^2}^2\\
&\leq C\|\tilde\vph^2|u_\th|^2\|_{W^{1,1}}^2\\
&\leq C\left[
\left(\int \tilde\vph^2|u_\th|^2\right)^2+
\left(\int\tilde\vph^2 (|u_{s\th}|+|u_{\th\th}|)|u_\th|\right)^2+
\left(\int\tilde\vph |u_\th|^2\right)^2
\right]\\
&\leq C\left[
\left(\int_{spt\tilde\vph}|u_\th|^2\right)^2+
\left(\int_{spt\tilde\vph}|u_\th|^2\right)
\left(\int \tilde\vph^4(|u_{s\th}|^2+|u_{\th\th}|^2)\right)
\right].\label{mother_split_est}
\eeqa
In order to control the right-hand side in terms of $\Th$, we require some extra powers of $\tilde \vph$ in the integrands. We achieve this by observing that if $s$ is within the support of $\tilde\vph$, then we have $s\in [s_0-1,s_0+1]$, and hence
\beqs
1\leq \tilde\vph^4(s-\half)+\tilde\vph^4(s+\half).
\eeqs
Therefore, assuming now that $|s_0|\leq X(\ell)-3/2$
(to prevent the integrals falling off the end of $\Col(\ell)$) we have
\beq \label{split-theta}
\int_{spt\tilde\vph}|u_\th|^2 \leq \Th(s_0-\half)+\Th(s_0+\half).
\eeq
Applied to only one factor of $\int_{spt\tilde\vph}|u_\th|^2$ in \eqref{mother_split_est}, one consequence
of this estimate is
\beqs
\int\tilde\vph^4|u_\th|^4 \leq
C
\left(\Th(s_0-\half)+\Th(s_0+\half)\right)
\left[
E(u;\Cyl_1(s_0))
+
\int_{\Cyl_1(s_0)}(|u_{s\th}|^2+|u_{\th\th}|^2).
\right]
\eeqs
\newcommand{\Lep}{{\Lambda_{\eps_2}}}
Hence by Corollary \ref{cor6}, for arbitrary $\ep_2$ (to be picked later) there exists $\Lep<\infty$ such that 
if $|s_0|\leq X(\ell)-\Lep$, then
\newcommand{\Cep}{C_{\eps_2}}
\beq
\label{last_term_est}
\int\tilde\vph^4|u_\th|^4 \leq
C\ep_2
\left(\Th(s_0-\half)+\Th(s_0+\half)\right)
+
\Cep\|\tau(u)\|_{L^2(\Cyl_\Lep(s_0))}^2
\eeq
where $C$ is not allowed to depend on $\ep_2$, but $\Cep$ is.
The estimate above will be used to control the final term in \eqref{Th_deriv}. We now wish to control the penultimate term, and thus estimate, using again \eqref{mother_split_est},
but now applying \eqref{split-theta} to all factors of $\int_{spt\tilde\vph}|u_\th|^2 $
\beqa
\label{geyser}
&\int\tilde\vph^4|u_\th|^2|u_s|^2
\leq 
\left(
\int\tilde\vph^4|u_\th|^4
\right)^\half
\left(
\int\tilde\vph^4|u_s|^4
\right)^\half\\
&\hspace{1.5cm}\leq
C
\left(\Th(s_0-\half)+\Th(s_0+\half)\right)
\left(
\int\tilde\vph^4|u_s|^4
\right)^\half\\
&\hspace{1.5cm}\quad 
+
C
\left(\Th(s_0-\half)+\Th(s_0+\half)\right)^\half
\left(
\int\tilde\vph^4|u_s|^4
\right)^\half
\left(\int \tilde\vph^4(|u_{s\th}|^2+|u_{\th\th}|^2)\right)^\half\\
&\hspace{1.5cm}\leq
C
\left(\Th(s_0-\half)+\Th(s_0+\half)\right)
\left[
\left(
\int\tilde\vph^4|u_s|^4
\right)^\half
+
\left(
\int\tilde\vph^4|u_s|^4
\right)
\right]\\
&\hspace{1.5cm}\quad
+\frac{1}{4}\int \tilde\vph^4(|u_{s\th}|^2+|u_{\th\th}|^2).
\eeqa
To control the part in square brackets, we use Young's inequality to estimate
$$
\left(
\int\tilde\vph^4|u_s|^4
\right)^\half
\leq
\sqrt{\ep_2}+\frac{1}{\sqrt{\ep_2}}
\int\tilde\vph^4|u_s|^4
$$
(with the same value of $\ep_2$ as above, still to be chosen)
and then use Corollary \ref{cor6} to control
$$
\int\tilde\vph^4|u_s|^4\leq \int_{\Cyl_1(s_0)}|u_s|^4
\leq 
\ep_2+\Cep\|\tau(u)\|_{L^2(\Cyl_{\Lep}(s_0))}^2,$$
$\Lep$ as before,
and we again require $|s_0|\leq X(\ell)-\Lep$ so that
$\Cyl_\Lep(s_0)\subset \Col(\ell)$.
Therefore \eqref{geyser} becomes
\beqa
\label{penultimate_term_est}
\int\tilde\vph^4|u_\th|^2|u_s|^2
&\leq 
C
\left(\Th(s_0-\half)+\Th(s_0+\half)\right)
\left[
\sqrt{\ep_2}
+
\Cep
\|\tau(u)\|_{L^2(\Cyl_\Lep(s_0))}^2
\right]\\
&\quad
+\frac{1}{4}\int \tilde\vph^4(|u_{s\th}|^2+|u_{\th\th}|^2)\\
&\leq 
C\sqrt{\ep_2}
\left(\Th(s_0-\half)+\Th(s_0+\half)\right)
+\Cep
\|\tau(u)\|_{L^2(\Cyl_\Lep(s_0))}^2
\\
&\quad
+\frac{1}{4}\int \tilde\vph^4(|u_{s\th}|^2+|u_{\th\th}|^2)
\eeqa
where $\Cep$ is allowed to depend on both $E_0$ and $\ep_2$ but again, $C$ does not depend on $\ep_2$.
Combining \eqref{Th_deriv} with our estimates
\eqref{last_term_est} and \eqref{penultimate_term_est},
and choosing $\ep_2>0$ sufficiently small, gives
us 
\beqas
\lefteqn{\Theta''(s_0)
+\frac{1}{8}\left(\Th(s_0-\half)+\Th(s_0+\half)\right)}
\qquad\qquad\qquad\\
&\geq
\int\tilde\varphi^4
\left(
(2-\frac12)(|u_{s\th}|^2+|u_{\th\th}|^2)
\right)
-C\|\tau(u)\|_{L^2(\Cyl_\Lep(s_0))}^2,
\eeqas
and by Wirtinger's inequality (i.e. the Poincar\'e inequality in one dimension)
$$\int_{\{s\}\times S^1}|u_{\th\th}|^2d\th\geq 
\int_{\{s\}\times S^1}|u_{\th}|^2d\th,$$
and so
\beqs
\Theta''(s_0)
+\frac32 \Th(s_0)
+\frac{1}{8}\left(\Th(s_0-\half)+\Th(s_0+\half)\right)
\geq
-C\|\tau(u)\|_{L^2(\Cyl_\Lep(s_0))}^2.
\eeqs
Now 
that
$\ep_2$, and hence $\La=\Lep$ has been fixed, we choose $\de>0$ sufficiently small so that whenever $s_0\in (-X_\de,X_\de)$ we automatically have $|s_0|\leq X(\ell)-\La$.
Note that unless $\ell>0$ is small enough, 
the claim will be vacuous because the collar $\Col(\ell)$ will contain no points of injectivity radius less than $\de$.
Also,
for $s\in \Cyl_\La(s_0)$, we have
$$\rho(s)\leq C\rho(s_0),$$
thanks to \eqref{rho_equiv},
so that 
switching from the tension $\tau(u)$ computed with respect to the flat metric to $\tau_g(u)$ computed with respect to $g$ gives
$$\|\tau(u)\|_{L^2(\Cyl_\Lambda(s_0))}^2
=\|\rho\,\tau_g(u)\|_{L^2(\Cyl_\Lambda(s_0),g)}^2\leq
\rho(s_0)^2\|\tau_g(u)\|_{L^2(\Cyl_\Lambda(s_0),g)}^2,$$
$g$ the hyperbolic metric on the collar, from which we conclude the lemma. 
\end{proof}

Equipped with Proposition \ref{prop:angular-energy}, we are finally in a position to prove the main estimate for the weighted angular energy.

\begin{proof}[Proof of Lemma \ref{ang_en_est}]
For the given, $N$, $E_0$ and map $u$, apply Proposition
\ref{prop:angular-energy} to give $C$, $\Lambda$ and $\de$ such that 
$$\int_{s_0-\frac12}^{s_0+\frac12}\int_{S^1}\abs{u_\theta}^2 d\theta ds\leq C\cdot e^{-(X(\ell)-\abs{s_0})} 
+C\cdot \int_{-X_\de}^{X_\de}e^{-\abs{s-s_0}}\rho^2(s)\cdot \norm{\tau_g(u)}_{L^2(\Cyl_\Lambda(s),g)}^2 ds$$
for all $s_0\in (-X_\de,X_\de)$. 
Multiplying by $\rho^{-2}(s_0)$ and integrating, we find that
\beqa
\label{wisteria}
\int_{-X_\de}^{X_\de}\int_{S^1} \rho^{-2}(s)|u_\th|^2(s,\th)d\th ds &\leq 
C\int_{-X_\de}^{X_\de}\rho^{-2}(s_0)\int^{s_0+1/2}_{s_0-1/2}
\int_{S^1}|u_\th|^2(s,\th)d\th\, ds\, ds_0\\
&\ \!\!\!\!\!\!\!\!\!\!\!\!\!\!\!\!\!\!\!
\leq 
C\int_{-X_\de}^{X_\de}\rho^{-2}(s_0)e^{-(X(\ell)-\abs{s_0})}ds_0\\
&\ \!\!\!\!\!\!\!\!\!\!\!\! +
C\int_{-X_\de}^{X_\de}\int_{-X_\de}^{X_\de}e^{-\abs{s-s_0}}\rho^{-2}(s_0)\rho^2(s) \norm{\tau_g(u)}_{L^2(\Cyl_\Lambda(s),g)}^2 ds
\,ds_0,
\eeqa
where we have appealed to \eqref{rho_equiv} to see that
for $s\in (s_0-\half,s_0+\half)$, we have
$\rho^{-2}(s)\leq C \rho^{-2}(s_0)$, or equivalently 
$\rho(s_0)\leq C\rho(s)$.

{\bf Claim:} For any $q\in [-X(\ell),X(\ell)]$, the function
$f(s):= \rho^{-2}(s)e^{-2|s-q|/3}$ for $s\in [-X(\ell),X(\ell)]$ is maximised when $s=q$. That is,
\beq \label{rho-monotone} 
\rho^{-2}(s)e^{-2|s-q|/3}\leq \rho^{-2}(q).\eeq

To see the claim, take logarithms, and differentiate, estimating using \eqref{rho_deriv} that for $s<q$ we have
$(\log f)'\geq -\frac{2}{\pi}+\frac{2}{3}>0$, whereas for $s>q$ we have $(\log f)'\leq +\frac{2}{\pi}-\frac{2}{3}<0$ as required.

Now, the first of the final two terms on the right-hand side 
of \eqref{wisteria} is bounded, independently of $\ell$, since 
by the claim we have
$$\rho^{-2}(s_0)e^{-(X(\ell)-\abs{s_0})}\leq 
\rho^{-2}(X(\ell))e^{-(X(\ell)-\abs{s_0})/3}\leq
\pi^2 e^{-(X(\ell)-\abs{s_0})/3}$$
for $\ell\in (0,2\arsinh(1))$ by \eqref{rho_range}.

To handle the final term on the right-hand side 
of \eqref{wisteria}, we can apply the claim again to find that
$$\int_{-X_\de}^{X_\de}e^{-\abs{s-s_0}}\rho^{-2}(s_0)ds_0
\leq \rho^{-2}(s)
\int_{-X_\de}^{X_\de}e^{-\abs{s-s_0}/3}ds_0
\leq C\rho^{-2}(s),$$
and hence we can improve \eqref{wisteria} to
\beqa
\label{wisteria2}
\int_{-X_\de}^{X_\de}\int_{S^1} \rho^{-2}(s)|u_\th|^2(s,\th)d\th ds &\leq 
C+
C\int_{-X_\de}^{X_\de}\norm{\tau_g(u)}_{L^2(\Cyl_\Lambda(s),g)}^2 ds\\
&\leq
C\left(1+\norm{\tau_g(u)}_{L^2(\Col(\ell),g)}^2\right),
\eeqa
because $\Cyl_\Lambda(s)\subset\Col(\ell)$ for all $s\in (-X_\de,X_\de)$.

To complete the proof, note that on the $\delta\thick$ part of the collar, described by 
$\{(s,\th): \abs{s}\in[X_\de,X(\ell))\}$,
the weight function is bounded by 
$\rho^{-2}\leq  \rho^{-2}(X_\de)\leq\pi^2\delta^{-2}$, compare \eqref{est:rho-inj-appendix}, 
so we may estimate 
$$\left(\int_{-X(\ell)}^{-X_\de}
+\int_{X_\de}^{X(\ell)}\right)
\int_{S^1} \rho^{-2}(s)|u_\th|^2(s,\th)d\th ds 
\leq 2\pi^2\delta^{-2} E(u)$$
in terms of the total energy $E(u)\leq E_0$.

Combining this with \eqref{wisteria2} completes the proof.
\end{proof}

\section{Paths in Teichm\"uller space}
\label{length:evol:sect}

Our goal in this section is to prove Lemma \ref{lengthevol}.
The main issue is to understand in detail the structure of the space of holomorphic quadratic differentials as the underlying metric degenerates by pinching one or more collars.
Our treatment follows on from our work in \cite{RTZ, RT2}, emphasising the geometric analysis of the subject, 
and is particularly adapted to get refined estimates that help us understand the projection $P_g$.
However, the general area of Weil-Petersson geometry has been considered by many authors, and our work makes connection with previous work as we describe below, 
even if we work independently to it. 
Example background references include Masur \cite{Masur}, Yamada \cite{Yamada1, Yamada2}, Wolpert \cite{Wolpert03, Wolpert07, Wolpert12} and the references therein. See also the recent work of Mazzeo and Swoboda \cite{maz_swo}.

As we explain below, a central part of our theory will be to make a decomposition of the space of holomorphic quadratic differentials on a surface, viewed as the tangent space to 
Teichm\"uller space, corresponding to a collection of geodesics having small lengths $\ell_i$. In addition to the kernel of $\{\partial \ell_i\}$ we effectively take the
dual basis to $\{\partial \ell_i\}$ on the orthogonal complement, as described in Remark \ref{geom_interpret}.
The refined $L^p$ estimates on this basis that we require in our applications are given in Lemma \ref{lemma:basis}.

\subsection{Structure of the space of holomorphic quadratic differentials}
\label{fouriersect}

We will need to understand the properties of holomorphic quadratic differentials on oriented closed hyperbolic surfaces, in particular in regions where the injectivity radius $\inj_g(p)$
is small. One fundamental fact of hyperbolic surface theory, 
cf. \cite{Hu}, Proposition IV.4.2, 
is that for any $0 <\delta < {\arsinh(1)}$, the $\delta$-thin part of the surface, consisting of all points at which the injectivity radius is less than $\de$, 
is given by a finite  union of disjoint hyperbolic cylinders of finite length around closed geodesics, which are explicitly described by the Collar lemma \ref{lemma:collar} 
of Keen-Randol. 
(Recall that these closed geodesics have length $\ell$, which we assume always to be less than $2\arsinh(1)$.)
Away from these collars, and more generally on the subset $\de\thick(M,g)$ of points $p$ with injectivity radius $\inj_{g}(p)\geq \de$, where $\de>0$, 
holomorphic quadratic differentials are well controlled. 
The most basic estimate, following 
from standard estimates for holomorphic functions on discs, is
that for $\Phi\in\Hol(M,g)$, $p\in [1,\infty)$
and $\de>0$, we have
\beq\label{est:holo-thick}
\norm{\Phi}_{{L^{\infty}(\de\thick(M,g))}}\leq C_\de\norm{\Phi}_{L^p(M,g)}\eeq
with $C_\de$ depending only on $\de$, 
cf. Lemma A.8 in \cite{RTZ}. While this estimate is sufficient for most of the paper, it is useful at times to observe 
that indeed
\beq\label{est:holo-thick-dehalf}
\norm{\Phi}_{{L^{\infty}(\de\thick(M,g))}}\leq C_\de\norm{\Phi}_{L^p(\frac\de2\thick(M,g))}\eeq
for any $0<\de<\arsinh(1)$, still with a constant $C_\de$ depending only on  $\de$. 

The Fourier decomposition of holomorphic quadratic differentials $\Phi$ on each hyperbolic collar $(\Col,g)$ 
\beqs 
\label{Fourier_decomp}
\Phi=\bigg(\sum_{n=-\infty}^\infty b_n e^{ns}\,e^{in\th}\bigg)\cdot dz^2, \quad b_n\in\C\eeqs
where $dz=ds+id\th$ as before,
gives an $L^2(\Col,g)$-orthogonal decomposition of each such $\Phi$
into its principal part $b_0dz^2=b_0(\Phi)dz^2=b_0(\Phi,\Col)dz^2$ and its \emph{collar decay} part $\om^\perp(\Phi,\Col):=\Phi-b_0(\Phi)dz^2$ which,
by  \cite[Lemma 2.2]{RTZ},
satisfies the key estimate
\beq 
\label{RTZLemma2.2}
\norm{\om^\perp(\Phi,\Col)}_{L^{\infty}(\de\thin(\Col,g))}\leq C\de^{-2}e^{-\frac\pi\de}\norm{\om^\perp(\Phi,\Col)}_{L^2(\de_0\thick(\Col,g))},
\eeq
for some universal constants $C<\infty$ and $\de_0\in (0,\arsinh(1))$, and any
$\de\in (0,\de_0]$, cf. \cite[section 3]{Wolpert12}.
(Here we are writing $\de\thin(\Col,g):=\Col\intersect \de\thin(M,g)$, and similarly for $\de\thick$.)

On the other hand, when $\Phi\in\Hol(M,g)$ we can also control the right-hand side of 
\eqref{RTZLemma2.2}, and even the $L^\infty$ norm, as follows.
Since we have made an orthogonal decomposition of $\Phi$, we know that 
$\langle\Phi,dz^2\rangle_{L^2(\Col,g)}=b_0(\Phi)\|dz^2\|_{L^2(\Col)}^2$, and hence
$$|b_0(\Phi,\Col)|\leq \frac{\|\Phi\|_{L^1(\Col)}\|dz^2\|_{L^\infty(\Col)}}{\|dz^2\|_{L^2(\Col)}^2}\leq C\ell\|\Phi\|_{L^1(\Col)},$$
for universal $C$, by \eqref{sizes_on_collars}.

Consequently, noting that, by \eqref{sizes_on_collars} and \eqref{est:rho-inj-appendix}, 
$\norm{dz^2}_{L^\infty(\de_0\thick(\Col,g))}\leq 2\pi^2\delta_0^{-2}$ is bounded above independently of $\ell\in (0,2\arsinh(1))$, 
and recalling \eqref{est:holo-thick}, we have by the triangle inequality
\beqa
\label{dethickest}
\norm{\om^\perp(\Phi,\Col)}_{L^\infty(\de_0\thick(\Col,g))}
&\leq
\norm{\Phi}_{L^\infty(\de_0\thick(\Col,g))}
+\norm{b_0(\Phi)dz^2}_{L^\infty(\de_0\thick(\Col,g))}\\
&\leq 
C\norm{\Phi}_{L^1(M,g)}
\eeqa
for universal $C<\infty$.
Using the fact that collars have uniformly bounded area as $\ell\downto 0$ (see \cite[Lemma A.5]{RTZ}), the norm on the right-hand side of \eqref{RTZLemma2.2} 
is controlled by the left-hand side of \eqref{dethickest}, and we find that
\beq 
\label{est-collar-decay}
\norm{\om^\perp(\Phi,\Col)}_{L^{\infty}(\de\thin(\Col,g))}\leq C\de^{-2}e^{-\frac\pi\de}\norm{\Phi}_{L^1(M,g)},
\eeq
for universal $C<\infty$, and any
$\de\in (0,\de_0]$.
Thus, deep within a collar, only the principal part of a holomorphic quadratic differential is significant.

In the special case that $\de=\de_0$, estimate \eqref{est-collar-decay} 
can be added to \eqref{dethickest} to give a bound on the whole collar of 
\beq
\label{wholecollar}
\norm{\om^\perp(\Phi,\Col)}_{L^{\infty}(\Col,g)}\leq 
C\norm{\Phi}_{L^1(M,g)},
\eeq
for a universal constant $C<\infty$.

Remark that the above estimates could also be derived based on the observation that $\om^\perp$ and $b_0dz^2$ are orthogonal 
also on subcylinders
of the collar, in particular on the $\de_0\thick$-part. Indeed, combined with \eqref{RTZLemma2.2} and 
\eqref{est:holo-thick-dehalf} this gives the slight improvement of \eqref{est-collar-decay} that
\beq 
\label{est-collar-decay-withde0}
\norm{\om^\perp(\Phi,\Col)}_{L^{\infty}(\de\thin(\Col,g))}\leq C\de^{-2}e^{-\frac\pi\de}\norm{\Phi}_{L^1(\frac{\de_0}2\thick(M,g))},
\eeq
still with universal $C$ and any $\de\in (0,\de_0]$, which will be needed in the proof of Lemma \ref{lemma:basis} later on. 

Following \cite{RTZ}, given an oriented closed hyperbolic surface $(M,g)$ on which we identify $k$ collars $\Col_1,\ldots,\Col_k$, we can then define the subspace 
\beqs 
W:=\{\Th\in\Hol(M,g):\, b_0(\Th,\Col_j)dz^2=0 \text{ for every } j\in\{1\ldots k\}\,\}\eeqs
of all holomorphic quadratic 
differentials with principal part equal to zero on each of the $k$ collars.

\begin{rmk}
\label{kernel_rmk}
When we return to viewing the space of holomorphic quadratic differentials $\Hol(M,g)$ on $(M,g)$ as tangent vectors in $\m_{-1}$ at $g$, via the isomorphism $\Phi\mapsto Re(\Phi)$, the subspace $W$ will have a simple 
geometric interpretation (albeit one that we will not require in any argument below). 
Indeed, if the lengths of the geodesics at the centre of the collars 
$\Col_1,\ldots,\Col_k$ are given as $\ell_1,\ldots,\ell_k$ near to $g$ in $\m_{-1}$, then $W$ will represent the kernel of 
$(\partial\ell_1,\ldots,\partial\ell_k)$ within $\Hol(M,g)$.
Note that here we view $\Hol$ as a real vector space with a complex structure $J$, and write
$\partial \ell_j:= (d\ell_j - iJd\ell_j)/2$, acting on $\Hol$ within its complexification. Thus, an element $v\in\Hol$ lies in the kernel of $\partial \ell_j$ precisely if both $v$ and $Jv$ lie in the kernel of $d\ell_j$.
See also Remarks \ref{simple_formula_holo_case} and \ref{geom_interpret}.
\end{rmk}

The following lemma is part of \cite[Lemma 2.4]{RTZ}, cf. \cite{Masur}.

\begin{lemma}[Decomposition of $\Hol(M,g_n)$: Definition of $W_n$]\label{W_structure}
Given an oriented closed surface $M$ of genus $\ga\geq 2$, there exists a 
universal constant $C$ such that the following is true.
Suppose $g_n$ is a sequence of hyperbolic metrics on $M$, and
apply the differential geometric Deligne-Mumford compactness 
Proposition \ref{Mumford}
in order to pass to a subsequence and obtain precisely 
$k\geq 0$ collars $\Col^j_n$ degenerating.
After omitting finitely many terms, the subspace
\beq \label{def:Wm} W_n:=\left\{\Th\in\Hol(M,g_n):\, b_0(\Th,\Col_n^j)dz^2=0 \text{ for every } j\in\{1\ldots k\}\right\}\eeq
of holomorphic quadratic differentials that have vanishing principal part on every collar $\Col^j_n$, $j\in \{1\ldots k\}$, is of (complex) dimension $3(\gamma-1)-k$,
and elements $w\in W_n$ decay rapidly along the collars in the sense that for sufficiently small $\delta>0$ (so that the 
$\de\thin$ part of $(M,g_n)$ lies within 
$\union_{j=1}^k\Col_n^j$ for every $n$) independent of $n$,
we have 
\beq
\label{w_decay}
\|w\|_{L^\infty(\de\thin(M,g_n))}\leq
C\cdot \delta^{-2}e^{-\pi/\delta}
\|w\|_{L^2(M,g_n)}.
\eeq

Furthermore, the spaces $W_n$ converge to $\Hol(\Si,h)$,  the space of integrable holomorphic quadratic differentials
on the typically noncompact limit surface 
$(\Si,h)$ from Proposition \ref{Mumford}, in the sense that for every 
$w\in\Hol(\Si,h)$ there exists a sequence $w_n\in W_n$ with  
\beq \label{conv:wn}f_n^*w_n\to w \text{ smoothly locally on } \Si \text{ and } \norm{w}_{L^2(\Si,h)}=\lim_{n\to\infty}\norm{w_n}_{L^2(M,g_n)}.\eeq
\end{lemma}

Note that \eqref{w_decay} also  follows from \eqref{est-collar-decay}, albeit with a constant $C$ now depending on the genus.
Note also that a combination of \eqref{w_decay} and \eqref{est:holo-thick} tells us that we can also control 
\beq
\label{w_decay2}
\|w\|_{L^\infty(M,g_n)}\leq
C\cdot \|w\|_{L^2(M,g_n)},
\eeq
for all $w\in W_n$, or, via \eqref{est-collar-decay}, alternatively with the $L^1$ norm on the right-hand side (with $C$ depending only on 
$\inf\{\inj_{g_n}(x):\, x\in 
M\backslash\union_{j=1}^k\Col_n^j,\, n\in\N\}$ i.e. with $C$ independent of $n$).

\begin{cor}
\label{W_proj_est}
In the setting of Lemma \ref{W_structure}, there exists a  constant $C<\infty$ such that for 
any $n$ and any $\Psi\in\Qu_{L^2}(M,g_n)$, we have
\beq
\label{was4.14}
\|P_{g_n}^{W_n}(\Psi)\|_{L^\infty(M,g_n)}\leq C\|\Psi\|_{L^1(M,g_n)}.
\eeq
\end{cor}

\begin{proof}
Estimate \eqref{w_decay2} implies
$$\|P_{g_n}^{W_n}(\Psi)\|_{L^\infty(M,g_n)}\leq C
\|P_{g_n}^{W_n}(\Psi)\|_{L^2(M,g_n)}.$$
On the other hand, by definition,
$$\|P_{g_n}^{W_n}(\Psi)\|_{L^2(M,g_n)}^2
=\langle \Psi,P_{g_n}^{W_n}(\Psi)\rangle\leq
\|\Psi\|_{L^1(M,g_n)}\|P_{g_n}^{W_n}(\Psi)\|_{L^\infty(M,g_n)}.$$
Combining these two estimates gives the result.
\end{proof}

Before we discuss the structure of the $L^2$-orthogonal complement $W_n^\perp$, we briefly discuss the
size of both the principal and the collar decay part of holomorphic quadratic differentials on the \textit{$\de$-thick} part of a collar, for \emph{fixed} $\de\in  (0,\arsinh(1))$.
Since we have made an orthogonal decomposition we know that for any $\Phi\in\Hol(M,g)$ and any collar $\Col=\Col(\ell)$
in $(M,g)$, we have
\beq 
\label{est:apriori-upper-bound-b0}
\abs{b_0(\Phi,\Col)}\leq \norm{dz^2}_{L^2(\Col,g)}^{-1}\norm{\Phi}_{L^2(M,g)}=\left(\frac{32\pi^5}{\ell^3}-\frac{16\pi^4}{3}+O(\ell^2)\right)^{-1/2}\norm{\Phi}_{L^2(M,g)},
\eeq
by \eqref{sizes_on_collars}. 
Since $\abs{dz^2}= 2\rho^{-2}$ is bounded above (and below) 
on $\de\thick(\Col,g)$ independently of $\ell$ -- see \eqref{est:rho-inj-appendix} and \eqref{est:rho-by-inj} -- this means in particular that 
for any $L^2$-unit element $\Phi\in \Hol(M,g)$ we have
\beq
\label{b0_2}
\abs{b_0(\Phi,\Col)dz^2}\leq C\ell^{3/2}
\eeq
on the $\de$-thick part of the collar, where $C$ depends only on $\de$. 

Meanwhile, on this thick part of $\Col(\ell)$ the \emph{collar decay} part of a unit element can be of order $1$, i.e. huge compared with
the principal part.

The following lemma guarantees that such a difference in size does not occur for elements of the orthogonal complement of $W_n$;
we obtain a basis of $W_n^\perp$ with each element concentrated on one of the degenerating collars 
and there having principal part almost as large as it can be according 
to \eqref{est:apriori-upper-bound-b0}, and with the remaining collar decay part now only of order $O(\ell^{3/2})$, i.e. of the same order as the principal part on the thick
part of the surface.

\begin{lemma}[Decomposition of $\Hol(M,g_n)$: Decomposition of 
$W_n^\perp$]
\label{Wnperp}
Suppose we are in the setting of Lemma \ref{W_structure}, and have identified the subspace $W_n$ of complex dimension $3(\ga-1)-k$.
Then there exist a constant $C\in (0,\infty)$, and for all $\de>0$
a further constant $C_\de\in (0,\infty)$,
such that for sufficiently large $n$
we can find an $L^2$-orthonormal basis $\{\Om_n^j\}$, $1\leq j\leq k$ of the $L^2$-orthogonal complement $W_n^\perp$ of $W_n$ in $\Hol(M,g_n)$, 
with one element concentrated on each collar  $\Col_n^j$ in the sense that 
for each 
$j\in\{1,\ldots,k\}$,
\beq\label{est:Om-complement}
\norm{\Om_n^j}_{L^\infty(M\setminus \Col_n^j,g_n)}\leq C(\ell_n^j)^{3/2},
\eeq
with the principal parts on the other collars $\Col_n^i$ 
controlled by the stronger bounds
\beq 
\label{est:princ_different_collar}
\abs{b_0(\Om_n^j,\Col_n^i)}\leq C (\ell_n^j)^{3/2}(\ell_n^i)^3 ,\quad  i\neq j.
\eeq
On the one collar $\Col_n^j$ where $\Om_n^j$ concentrates it is
essentially given as a constant multiple of $dz^2$ in the sense that
\beq\label{est:Om-collar} 
\norm{\Om_n^j-\beta_n^j dz^2}_{L^\infty(\Col_n^j,g_n)}\leq C (\ell_n^j)^{3/2} \text{ where } \bigg|\frac{(\ell_n^j)^{3/2}}{(32\pi^5)^{1/2}}-\beta_n^j\bigg|\leq C(\ell_n^j)^{9/2},
\eeq
for $\beta_n^jdz^2=b_0(\Om_n^j,\Col_n^j)dz^2$ the principal part on $\Col_n^j$ and $\ell_n^j$ the length of the central geodesic in $(\Col_n^j,g_n)$. 
Moreover, for each $j\in\{1,\ldots,k\}$, we have
\beq
\label{L1Omest}
\|\Om_n^j\|_{L^1(M,g_n)}\leq C(\ell_n^j)^{1/2},
\eeq
\beq
\label{LinfOmest}
\norm{\Om_n^j}_{L^\infty(M,g_n)}\leq 
C(\ell_n^j)^{-1/2},
\eeq
and
\beq
\label{extra_control_on_si}
\norm{\Om_n^j}_{L^\infty(\de\thick(M,g_n))}\leq C_\de(\ell_n^j)^{3/2}.
\eeq
\end{lemma}

This lemma will be derived from the following lemma which is an improvement of \cite[Lemma 2.6 and Corollary 2.7]{RT2} and which 
gives an explicit definition of basis elements concentrating
on collars, but does not guarantee orthogonality 

\begin{lemma}\label{lemma:basis}
Suppose we are in the setting of Lemma \ref{W_structure}, and have identified the subspace $W_n$ of complex dimension $3(\ga-1)-k$.
Then there exist a constant $C\in (0,\infty)$, and for all $\de>0$
a further constant $C_\de\in (0,\infty)$, 
such that for sufficiently large $n$, 
the $L^2$-unit holomorphic quadratic differentials $\widetilde\Om_n^1,\ldots,\widetilde\Om_n^k\in W_n^\perp$
characterised by
\beq\label{def:Omw}
b_0(\widetilde\Om_n^j,\Col_n^i)=0 \text{ for every } i\neq j \text{ while } b_0(\widetilde\Om_n^j,\Col_n^j)>0
\eeq
concentrate on the respective collars $\Col_n^1,\ldots,\Col_n^k$ in the sense that for each $j\in\{1,\ldots,k\}$, we have
\beq\label{est:tOm-complement}
\norm{\widetilde\Om_n^j}_{L^\infty(M\setminus \Col_n^j,g_n)}\leq C(\ell_n^j)^{3/2},
\eeq
while on this collar $\widetilde\Om_n^j$ is essentially given as a constant multiple of $dz^2$ in the sense that 
\beq\label{est:tOm-collar} 
\norm{\widetilde\Om_n^j-\tilde\beta_n^j dz^2}_{L^\infty(\Col_n^j,g_n)}\leq C (\ell_n^j)^{3/2},
\eeq
for the principal part $\tilde\beta_n^jdz^2=b_0(\widetilde\Om_n^j,\Col_n^j)dz^2$ of $\widetilde\Om_n^j$ 
satisfying 
\beq
\label{b_estimate}
1-C(\ell_n^j)^3\leq \tilde\beta_n^j \norm{dz^2}_{L^2(\Col_n^j,g_n)}\leq 1.
\eeq
Moreover, the elements $\widetilde\Om_n^j$ are almost orthogonal,
\beq \label{est:Om-orthogonal}
\langle \widetilde\Om_n^j,\widetilde\Om_n^i\rangle_{L^2(M,g)}\leq C(\ell_n^j)^{3/2}(\ell_n^i)^{3/2}, \qquad i\neq j
\eeq 
and for each $j\in\{1,\ldots,k\}$, we have
\beq
\label{L1claimnew}
\|\widetilde\Om_n^j\|_{L^1(M,g_n)}\leq C(\ell_n^j)^{1/2},
\eeq
\beq
\label{tLinfOmest}
\norm{\widetilde\Om_n^j}_{L^\infty(M,g_n)}\leq 
C(\ell_n^j)^{-1/2},
\eeq
and
\beq \label{est:key-Omw}
\norm{\widetilde\Om_n^j}_{L^1(\delta\thick(M,g_n))}\leq 
C_\de (\ell_n^j)^{3/2}.
\eeq
\end{lemma}

\begin{rmk}
\label{rem:b_value}
Note that by virtue of \eqref{sizes_on_collars} and \eqref{b_estimate}, the constants $\tilde\be_n^j$ are constrained by
\beq
\label{b_value}
\bigg|\tilde\be_n^j-\frac{(\ell_n^j)^{3/2}}{(32\pi^5)^{1/2}}\bigg|
\leq C(\ell_n^j)^{9/2},
\eeq
in particular 
\beqs 
\left\|\widetilde \Om_n^j-\frac{(\ell_n^j)^{3/2}}{(32\pi^5)^{1/2}}dz^2\right\|_{L^\infty(\Col_n^j,g_n)}\leq C (\ell_n^j)^{3/2}.
\eeqs
\end{rmk}

\begin{rmk}
With more effort, the dependencies of the constants $C$, $C_\delta$ in Lemma \ref{lemma:basis} can be dramatically improved, although we do not require such refinements in this paper.
For example, by working directly, one can show that the constants in \eqref{L1claimnew} and \eqref{tLinfOmest} can be chosen to be universal (cf. \cite{RT2}).
Improved dependencies in Lemma \ref{lemma:basis} then yield improved dependencies in Lemma \ref{Wnperp}.
\end{rmk}

\begin{rmk}
\label{geom_interpret}
For $n$ so large that the collars have been almost pinched, the basis 
$\{\Om_n^j\}$ of $W_n^\perp$ will be very similar to the basis
$\{\widetilde\Om_n^j\}$. A further similar basis in this limit, see Wolpert \cite{Wolpert07}, can be given in terms of the scaled Weil-Petersson gradients $\{-\sqrt\frac{\pi}{2\ell_n^j}\grad\ell_n^j\}$.
In contrast, our basis  $\{\widetilde\Om_n^j\}$ can be viewed as the dual basis to $\{\partial\ell_n^j\}$, with each element scaled to unit length. 
Here $\partial\ell^j_n=(d\ell^j_n - iJd\ell^j_n)/2$ is like in Remark \ref{kernel_rmk} except that we restrict it to the complement $W_n^\perp$ of the kernel of
$(\partial \ell^1_n,..,\partial \ell^k_n)$
within $\Hol(M,g_n)$. One should also compare with the work of Masur \cite{Masur}.

\end{rmk}

\begin{rmk}
\label{A6}
Although we will not need them, we remark that based on \eqref{est:tOm-collar} and \eqref{est:tOm-complement} we get refinements of \eqref{est:tOm-complement}
and \eqref{est:tOm-collar} at the \emph{central geodesics} of a collar. That is, 
for $i\neq j$, we have
\beqs
\norm{\widetilde\Om_n^j}_{L^\infty(\si_n^i,g_n)}\leq C(\ell_n^j)^{3/2}
\left((\ell_n^i)^{-2}e^{-\frac{2\pi}{\ell_n^i}}\right),
\eeqs
and for each $i$,
\beqs
\norm{\widetilde\Om_n^i-\tilde \be_n^i dz^2}_{L^\infty(\si_n^i,g_n)}\leq C (\ell_n^i)^{-1/2}e^{-\frac{2\pi}{\ell_n^i}}.
\eeqs
\end{rmk}

\begin{proof}[Proof of Lemma \ref{lemma:basis}]
As Lemma \ref{W_structure} has told us, after having omitted finitely many terms, the subspace 
$W_n^{\perp}$ has dimension $k$, and we may as well assume that $k\geq 1$, otherwise the lemma that we are proving is vacuous. 
Therefore, for each fixed $j\in\{1,\ldots,k\}$, there is a 
unique element $\widetilde\Om_n^j$ of $W_n^\perp$ with $\norm{\widetilde\Om_n^j}_{L^2(M,g_n)}=1$ 
satisfying \eqref{def:Omw}.
The key step in the proof of Lemma \ref{lemma:basis} is to show that claim \eqref{est:key-Omw} holds true for this basis of $W_n^\perp$.
Assume instead that there exists a number $\bar\de>0$ such that after passing to a subsequence
\beq 
\label{ass:La-to-contradict}
\La_n:=(\ell_n^j)^{-3/2}\norm{\widetilde \Om_n^j}_{L^1(\bar\de\thick(M,g_n))}\to \infty \quad\text{ as } n\to\infty,
\eeq
for some $1\leq j\leq n$, say for $j=1$. 
 
We now choose a sequence $\de_n\in (0,\bar\de]$ with $\de_n\to 0$ so that still
\beq 
\label{def:de-n} 
\La_n\cdot\de_n\to \infty \text{ as } n\to\infty\eeq
and set 
\beq \label{def:lambdan}
\lambda_n:=\norm{\widetilde\Omega_n^1}_{L^1(\de_n\thick(M,g_n))}\geq \La_n (\ell_n^1)^{3/2}.\eeq

We then consider the normalised sequence 
$$\widehat{\Om}_n:= (\lambda_n)^{-1}\cdot\widetilde\Om_n^1,$$ 
so
\beq
\label{Omhatnormalisation}
\norm{\widehat\Om_n}_{L^1(\de_n\thick(M,g_n))}=1,
\eeq
and prove the following two claims which obviously contradict each other.

{\bf Claim 1:} [Not everything disappears down the collar.] There exists a number $\de>0$ such that for all $n$ sufficiently large
\beqs 
\norm{\widehat\Omega_n}_{L^1(\de\thick(M,g_n))}\geq \frac12.
\eeqs

{\bf Claim 2:} [Everything disappears down the collar.] After passing to a subsequence and pulling back by the diffeomorphisms $f_i:\Si\to M\setminus \cup_{j=1}^k \sigma^{j}_{i}$ given by 
Proposition \ref{Mumford}, we have
\beq \label{claim2}
f_n^*\widehat\Omega_n\to 0 \text{ smoothly locally on } \Si,\eeq
which implies
$$\lim_{n\to\infty}\norm{\widehat\Omega_n}_{L^1(\de\thick(M,g_n))}=0\quad \text{ for every } \de>0.$$

\begin{proof}[Proof of Claim 1]
Recall that for sufficiently small $\de>0$, the $\de\thin$ part of $(M,g_n)$ is given as a  union of disjoint subsets of the degenerating collars $\Col_n^i$.
In view of the normalisation \eqref{Omhatnormalisation},
we thus need to show that (after possibly reducing $\de>0$)
\beq\label{est:claim1}
\sum_{i=1}^k\norm{\widehat\Omega_n}_{L^1\big(\de\thin(\Col_n^i)\setminus\de_n\thin(\Col_n^i)\big)} 
\leq \frac12
\eeq
for all $n$ sufficiently large.
By definition, the principal part of 
$\widetilde\Om_n^1$, and thus of $\widehat\Omega_n$, on $\Col_n^i$ vanishes
for all $i\neq 1$, so that \eqref{est-collar-decay-withde0} applies, resulting for 
$n$ large in an estimate of
\beqas 
\norm{\widehat\Omega_n}_{L^1(\de\thin(\Col_n^i)\setminus\de_n\thin(\Col_n^i))}
&\leq C \norm{\widehat\Omega_n}_{L^\infty(\de\thin(\Col_n^i))}
\leq C\de^{-2}e^{-\pi/\de}\norm{\widehat\Omega_n}_{L^1(\frac{\de_0}{2}\thick(M,g_n))}\\
&\leq \frac1{4k} \norm{\widehat\Omega_n}_{L^1(\de_n\thick(M,g_n))}=\frac1{4k}
\eeqas
provided $\de>0$ is initially chosen small enough.
The same estimate is of course valid also for the collar decay part $\om^\perp(\widehat\Omega_n,\Col_n^1)$. 

In order to prove \eqref{est:claim1} it thus remains to control the contribution of the principal part to the $L^1$ norm on $\de\thin(\Col_n^1)\setminus\de_n\thin(\Col_n^1)$.
Here we crucially use the assumption \eqref{ass:La-to-contradict} which means that 
the principal part of $\widetilde\Om_n^1$ is small compared to $\widetilde\Om_n^1$ on the 
$ \bar \de$-thick part of the surface. More precisely,
by the a priori bound 
\eqref{b0_2}
valid for all elements of $\Hol$, we know that
$$\frac{\norm{b_0(\widetilde\Om_n^1,\Col_n^1)dz^2}_{L^1(\bar\de\thick(\Col_n^1))}}{\norm{\widetilde\Om_n^1}_{L^1(\bar\de\thick(M,g_n))}}\leq \frac{C_{\bar\de}(\ell_n^1)^{3/2}}{\La_n(\ell_n^1)^{3/2}}
\to 0\text{ as } n\to\infty.$$
The relation \eqref{def:de-n} now guarantees that this property is preserved for our sequence $\de_n\to 0$: 
Indeed, by 
Proposition \ref{X_Xdelta}
and \eqref{sizes_on_collars}, we have
$$\norm{dz^2}_{L^1(\de_n\thick(\Col_n^1))}=8\pi(X(\ell_n^1)-X_{\de_n}(\ell_n^1))\leq C\de_n^{-1},$$
and combining this estimate with 
\eqref{est:apriori-upper-bound-b0}, \eqref{def:de-n} and \eqref{def:lambdan} yields
\beqas
\frac{\norm{b_0(\widetilde\Om_n^1,\Col_n^1)dz^2}_{L^1(\de_n\thick(\Col_n^1))}}{\norm{\widetilde\Om_n^1}_{L^1(\de_n\thick(M,g_n))}}
&=\frac{\abs{b_0(\widetilde\Om_n^1,\Col_n^1)}.\norm{dz^2}_{L^1(\de_n\thick(\Col_n^1))}}{\lambda_n}\\
&\leq\frac{C(\ell_n^1)^{3/2}}{\lambda_n\delta_n}\leq \frac{C}{\Lambda_n\de_n}
\to 0 \eeqas
as $n\to\infty$. But the left-hand side is precisely 
the $L^1$ norm of the principal part of the normalised $\widehat\Om_n$ on $\de_n\thick(\Col_n^1)$, so that
$$\norm{\widehat\Om_n}_{L^1(\de\thin(\Col_n^i)\setminus\de_n\thin(\Col_n^i))}\to 0$$
is certainly less than $\frac14$ for $n$ large, completing the proof of \eqref{est:claim1}, i.e. Claim 1.
\end{proof}

\begin{proof}[Proof of Claim 2]
Observe that while the $L^1$-norms of $\widehat\Om_n$ are in general unbounded, our normalisation \eqref{Omhatnormalisation} implies that 
for every $\de>0$ 
$$\limsup_{n\to\infty}\norm{\widehat\Om_n}_{L^1(\de\thick(M,g_n))}\leq 1.$$
To begin with, we remark that such a bound turns out to be sufficient
to extract a subsequence in $n$ so that 
$f_n^*\widehat\Om_n$ converges smoothly locally to a limit 
$\widehat\Om_\infty\in\Hol(\Si,h)$, $f_n$ the diffeomorphisms of the Deligne-Mumford compactness Proposition \ref{Mumford};
indeed, as $f_n^*\widehat\Om_n$ are holomorphic with respect to the complex structures $f_n^*c_n\to c$ their $C^k$ norm 
on balls of a given radius are bounded in terms of their $L^1$ norms on slightly larger balls, resulting in (uniform in $n$) bounds of
$$\norm{f_n^*\widehat\Om_n}_{C^k(f_n^*(\de\thick(M,g_n)))}\leq C_{k,\de}\norm{\widehat\Om_n}_{L^1(\de/2\thick(M,g_n))}\leq C_{k,\de}$$
for every $\de>0$, $k\in\N$.
Combined with the theorem of Arzela-Ascoli, this then allows us to extract a subsequence to give smooth local convergence 
$f_n^*\widehat\Om_n\to \widehat\Om_\infty$
to a holomorphic quadratic differential on 
$\Si$; for details of this argument we refer to Section A.3 of \cite{RTZ}, in particular to Lemma A.9 and its proof.
Since additionally
\beq 
\label{eq:L1-limit}
\norm{\widehat\Om_\infty}_{L^1(\de\thick(\Si,h))}=\lim_{n\to\infty}\norm{\widehat\Om_n}_{L^1(\de\thick(M,g_n))}\leq 1 \text{ for any }\de>0, \eeq
see \cite[Lemma A.7]{RTZ}, $  \widehat\Om_\infty$ is indeed an element of $\Hol(\Si,h)$, the space of all 
holomorphic quadratic differentials on the noncompact limit surface with finite $L^1(\Si,h)$ norm, or by \cite[Lemma A.11]{RTZ} equivalently, with
finite $L^2(\Si,h)$-norm. 
As we claim that $  \widehat\Om_\infty$ is identically zero we therefore need to prove that 
\beqs
\langle w,  \widehat\Om_\infty\rangle_{L^2(\Si,h)}= 0 \text{ for every } w\in\Hol(\Si,h).\eeqs

Given $w\in\Hol(\Si,h)$ we let $w_n\in W_n$ be an approximating sequence as in \eqref{conv:wn} and use that, by definition, $\widehat\Om_n$ is an element of $W_n^\perp$ 
so that for any $\de>0$
\beqa  \label{est:claim2-first-step}
\abs{\langle w,  \widehat\Om_\infty\rangle_{L^2(\Si,h)}}&\leq \abs{ \langle w, \widehat\Om_\infty\rangle_{L^2(\de\thick(\Si,h))}}
+\abs{ \langle w, \widehat\Om_\infty\rangle_{L^2(\de\thin(\Si,h))}}\\
&\leq \abs{\lim_{n\to\infty}\langle w_n,\widehat\Om_n\rangle_{L^2(\de\thick(M,g_n))}}+\norm{w}_{L^2(\de\thin(\Si,h))}\cdot\norm{\widehat\Om_\infty}_{L^2(\de\thin(\Si,h))}\\
&=\abs{\lim_{n\to\infty}\langle w_n,\widehat\Om_n\rangle_{L^2(\de\thin(M,g_n))}}+R(\de)
\eeqa 
where we observe that $R(\de)=\norm{w}_{L^2(\de\thin(\Si,h))}\cdot\norm{\widehat\Om_\infty}_{L^2(\de\thin(\Si,h))}\to0$ as $\de\to 0$ since 
both $w$ and $\widehat\Om_\infty$ have finite $L^2(\Si,h)$ norm. 

On the other hand, we recall that for $\de>0$ sufficiently small, 
the $\de$-thin part of $(M,g_n)$ is given as the union of the $\de$-thin parts of the degenerating collars $\Col_n^{i}$, and stress once more that the 
collar decay part of any holomorphic quadratic differential is $L^2$-orthogonal to $dz^2$ on arbitrary subcylinders of the collars. 
Since the principal part of $w_n$ on \textit{all}  degenerating collars $\Col_n^i$ is zero, 
the inner product in \eqref{est:claim2-first-step} is given only in terms of the 
collar decay parts of $ \widehat\Om_n$ so, by \eqref{est-collar-decay-withde0} and \eqref{w_decay}, for $n$ sufficiently large (so that $\de_n<\frac{\de_0}2$)
\beqas 
\abs{\langle w_n,\widehat\Om_n\rangle_{L^2(\de\thin(M,g_n))}}
&\leq \sum_{i=1}^k\norm{w_n}_{L^2(\de\thin(\Col_n^i))}\cdot \norm{\om^\perp(\widehat\Om_n,\Col_n^i)}_{ L^2(\de\thin(\Col_n^i))}\\
&\leq C\big(\de^{-2}e^{-\pi/\de}\big)^2\sum_{i=1}^k
\norm{w_n}_{L^2(M,g_n)}\cdot \norm{\widehat\Om_n}_{L^1(\frac{\de_0}2\thick(M,g_n))}\\
&\leq C\big(\de^{-2}e^{-\pi/\de}\big)^2.
\eeqas

Combined with \eqref{est:claim2-first-step} we thus find that 
$$\abs{\langle w,\widehat\Om_\infty\rangle_{L^2(\Si,h)}}\leq R(\de)+C\de^{-4}e^{-2\pi/\de}\to 0\text{ as } \de\to0$$
so that the limit $\widehat\Om_\infty$ obtained above must be zero, proving not only \eqref{claim2}, but at the same time, by 
\eqref{eq:L1-limit}, also the claimed $L^1$-bounds.
\end{proof}

This completes the proof of the key estimate \eqref{est:key-Omw}.

As for the remaining parts, \eqref{est:tOm-complement} 
follows from \eqref{est:key-Omw}, 
\eqref{est:holo-thick-dehalf} 
and 
\eqref{est-collar-decay-withde0},
if we choose $\de\in (0,\de_0]$ to be sufficiently small so that the $\de\thin$ part of $(M,g_n)$ is contained within $\union_{j=1}^k\Col_n^j$ for sufficiently large $n$.

Estimate \eqref{est:tOm-collar} follows from 
\eqref{est-collar-decay-withde0}, \eqref{est:key-Omw} and \eqref{b_value}.

To establish \eqref{b_estimate}, we compute
\beqas
1&=\|\widetilde\Om_n^j\|_{L^2(M,g_n)}^2=
\|\widetilde\Om_n^j\|_{L^2(M\backslash\Col_n^j,g_n)}^2+\|\widetilde\Om_n^j\|_{L^2(\Col_n^j,g_n)}^2\\
&=\|\widetilde\Om_n^j\|_{L^2(M\backslash\Col_n^j,g_n)}^2
+\|
\tilde\be_n^j
dz^2\|_{L^2(\Col_n^j,g_n)}^2
+\|\widetilde\Om_n^j-\tilde\be_n^j dz^2\|_{L^2(\Col_n^j,g_n)}^2.
\eeqas
This implies the second inequality of \eqref{b_estimate}, but also, when combined with \eqref{est:tOm-complement} and 
\eqref{est:tOm-collar}, it gives
$$1-\|\tilde\be_n^j
dz^2\|_{L^2(\Col_n^j,g_n)}^2
\leq C(\ell_n^j)^3.$$

To show \eqref{est:Om-orthogonal}
we use the orthogonality of the principal and collar decay parts as well as that $b_0(\widetilde\Om_n^i,\Col_n^j)=0$ for $i\neq j$ to see that the inner product depends only on terms that 
are small according to \eqref{est:tOm-collar} and \eqref{est:tOm-complement}, namely
\beqas 
\langle \widetilde\Om_n^i,\widetilde\Om_n^j\rangle_{L^2(M,g_n)}&= 
\langle \widetilde \Om_n^i-\widetilde\beta_n^i dz^2,\widetilde\Om_n^j\rangle_{L^2(\Col_n^i,g_n)}
+ \langle \widetilde \Om_n^i,\widetilde\Om_n^j-\tilde\beta_n^j dz^2\rangle_{L^2(\Col_n^j,g_n)}\\
& \quad +\langle \widetilde \Om_n^i,\widetilde\Om_n^j\rangle_{L^2(M\setminus(\Col_n^j\cup\Col_n^i),g_n)}\eeqas
which implies \eqref{est:Om-orthogonal}.

Finally, both  \eqref{L1claimnew} and \eqref{tLinfOmest}
follow from \eqref{est:tOm-complement}, \eqref{est:tOm-collar}, \eqref{b_value} and \eqref{sizes_on_collars}.
\end{proof}

\begin{proof}[Proof of Lemma \ref{Wnperp}.]
The orthonormal bases $\{\Om_n^j\}$ will arise as slight adjustments of the unit vectors $\{\widetilde\Om_n^j\}$
from Lemma \ref{lemma:basis}. To simplify notations, we fix $n$ and drop the subscript $n$.
By the estimate
\eqref{est:Om-orthogonal}
we can adjust $\{\widetilde\Om^j\}$ to an orthonormal basis $\{\Om^j\}$ inductively using Gram-Schmidt, setting $\Om^1=\widetilde\Om^1$ and 
$$\Om^j:=\left[\widetilde\Om^j-\sum_{i=1}^{j-1}
\langle\widetilde\Om^j,\Om^i\rangle\Om^i\right]\la_j^{-1},$$
for $j=2,\ldots,k$,
where $\la_j:=\|\widetilde\Om^j-\sum_{i=1}^{j-1}
\langle\widetilde\Om^j,\Om^i\rangle\Om^i\|_{L^2(M,g)}$.
Based on \eqref{est:Om-orthogonal} we may then prove by induction that
for $j=2,\ldots,k$, we may write
\beq
\label{changeofbasis}
\Om^j=\widetilde\Om^j\la_j^{-1}
+\sum_{i=1}^{j-1}c_{ji}\widetilde\Om^i,
\eeq
with 
\beq
\label{cest}
\begin{aligned}
|c_{ji}|&\leq C
(\ell^j \ell^i)^{3/2}\quad&\text{if } k\geq j>i\geq 1\\
1-\la_j^2&=\sum_{i=1}^{j-1} (\langle\widetilde\Om^j,\Om^i\rangle)^2\leq 
C(\ell^j)^3\sum_{i=1}^{j-1}(\ell^i)^3\quad&\text{ if } j\in\{2,\ldots,k\}.
\end{aligned}
\eeq

Because of \eqref{changeofbasis} and \eqref{cest}, we see that
\eqref{est:Om-complement}, \eqref{est:princ_different_collar}, \eqref{est:Om-collar},
\eqref{L1Omest}, \eqref{LinfOmest} and \eqref{extra_control_on_si}
hold for the orthonormal basis $\{\Om_j\}$ because of the analogous statements we proved for Lemma \ref{lemma:basis}.
For example, to prove \eqref{est:Om-complement}, we estimate
\beqas
\norm{\Om^j}_{L^\infty(M\setminus \Col^j,g)}
&=\norm{\widetilde\Om^j\la_j^{-1}+\sum_{i=1}^{j-1}c_{ji}\widetilde\Om^i}_{L^\infty(M\setminus \Col^j,g)}\\
&\leq
C\norm{\widetilde\Om^j}_{L^\infty(M\setminus \Col^j,g)}+
\sum_{i=1}^{j-1}|c_{ji}|.\norm{\widetilde\Om^i}_{L^\infty(M\setminus \Col^j,g)}\\
&\leq 
C(\ell^j)^{3/2}+
C\sum_{i=1}^{j-1}(\ell^j\ell^i)^{3/2}
(\ell^i)^{-1/2}
\\
&\leq C(\ell^j)^{3/2},
\eeqas
while \eqref{est:princ_different_collar} follows immediately from \eqref{changeofbasis}, \eqref{cest}, \eqref{b_value} and the definition of $\widetilde\Om^i$.
To prove \eqref{est:Om-collar} we use \eqref{cest} 
first to obtain the bound on
the principal part $\beta^jdz^2=\lambda_j^{-1}\tilde\beta^jdz^2$
$$\bigg|\beta^j-\frac{(\ell^j)^{3/2}}{(32\pi^5)^{1/2}}\bigg|\leq \bigg|\tilde\beta^j-\frac{(\ell^j)^{3/2}}{(32\pi^5)^{1/2}}\bigg|+\abs{\tilde\beta^j}.\abs{\la_j^{-1}-1}
\leq C(\ell^j)^{9/2}$$
from \eqref{b_value}, 
and then to derive the bound on the collar decay part
\beqas
\norm{\Om^j-
\beta^jdz^2
}_{L^\infty(\Col^j,g)}
&\leq 
\norm{\la_j^{-1}\left(\widetilde\Om^j-\tilde\beta^jdz^2\right)}_{L^\infty(\Col^j,g)}
+\sum_{i=1}^{j-1}|c_{ji}|.\norm{\widetilde\Om^i}_{L^\infty(\Col^j,g)}\\
&\leq C(\ell^j)^{3/2}+C(\ell^j)^{3/2}\sum (\ell^i)^{3/2}(\ell^i)^{3/2}\leq C(\ell^j)^{3/2}
\eeqas
from the corresponding bound \eqref{est:tOm-collar} on $\tilde \Om^j$ as well as from \eqref{est:tOm-complement}.

To prove \eqref{L1Omest}
we use \eqref{changeofbasis}, \eqref{cest} and \eqref{L1claimnew}.
Recalling how we proved \eqref{tLinfOmest}, we see that
\eqref{LinfOmest} follows immediately (for example) 
from \eqref{est:Om-complement}, \eqref{est:Om-collar} and
\eqref{sizes_on_collars}, while \eqref{extra_control_on_si} follows from \eqref{est:key-Omw} combined with \eqref{changeofbasis} and \eqref{cest}.
\end{proof}

\subsection{Projection of general quadratic differentials onto $\Hol$}

Based on the properties of holomorphic quadratic differentials derived in the previous section we can now prove:
\begin{prop}
\label{lemma:projection}
Given an oriented closed surface $M$ of genus $\ga\geq 2$, equipped with a hyperbolic metric $g$, there exists 
$C<\infty$ depending only on $\ga$ 
such that the projection $P_g(\Psi)$ of an arbitrary quadratic differential $\Psi\in \Qu_{L^2}(M,g)$ onto the space of holomorphic quadratic differentials satisfies
\beq \label{est:Linfty-Projection-wholesurface}
\norm{P_g(\Psi)}_{L^1(M,g)}\leq C\norm{\Psi}_{L^1(M,g)}.
\eeq
Moreover, on any collar $\Col= \Col(\ell)$ in $(M,g)$ with 
$\ell< 2\arsinh(1)$, we have
\beqs 
P_g(\Psi)\sim \frac{\ell^3}{32\pi^5}\langle\Psi,dz^2\rangle_{L^2(\Col,g)}dz^2\eeqs
in the sense that the principal part $b_0dz^2:=b_0(P_g(\Psi),\Col)dz^2$ of $P_g(\Psi)$ on $\Col$ satisfies
\beqs 
\label{est:principal-part-lemma}
\bigg|b_0-\frac{\ell^3}{32\pi^5}\langle\Psi,dz^2\rangle_{L^2(\Col)}\bigg|\leq C\cdot \ell^3\norm{\Psi}_{L^1(M,g)}\eeqs
while the remaining part decays rapidly along the collar, satisfying  
\beq \label{est:decay-part-lemma}
\norm{P_g(\Psi)-b_0dz^2}_{L^\infty(\delta\thin(\Col,g))}\leq C\delta^{-2}e^{-\pi/\delta}\norm{ \Psi}_{L^1(M,g)}  
\quad\forall
\delta\in(0,\arsinh(\cosh(\ell/2))).
\eeq
\end{prop}

The upper bound for $\de$ in \eqref{est:decay-part-lemma} could be taken to be any fixed number. The upper bound we chose is the injectivity radius at the ends of $\Col$ -- see \eqref{inj_at_end}.

\begin{proof}
Suppose, contrary to the proposition, that \eqref{est:Linfty-Projection-wholesurface} is false. Then there exists an  oriented closed surface $M$, a sequence of metrics $g_n\in \M$, and a sequence of elements $\Psi_n \in \Qu_{L^2}(M,g_n)$ such that
\beq
\label{to_contradict}
\norm{P_{g_n}(\Psi_n)}_{L^1(M,g_n)}> n\norm{\Psi_n}_{L^1(M,g_n)}.
\eeq
Lemmata \ref{W_structure} and \ref{Wnperp}
give us a subsequence and a decomposition 
$\Hol(M,g_n)=W_n\oplus W_n^\perp$
and allow us to write
\beqs
\label{eq:PnPsin}
P_{g_n}(\Psi_n):=w_n+\sum_{i=1}^k \langle\Psi_n,\Om_n^i\rangle\Om_n^i,\eeqs
where $w_n=P^{W_n}_{g_n}(\Psi_n)\in W_n$.
Corollary \ref{W_proj_est} and the fact that 
by the Gauss-Bonnet theorem
the area of $(M,g_n)$ is independent of $n$ gives
\beq 
\label{est:wn_est} 
\norm{w_n}_{L^1(M,g_n)}\leq C\norm{w_n}_{L^\infty(M,g_n)}\leq C\norm{\Psi_n}_{L^1(M,g_n)}.
\eeq
On the other hand, \eqref{L1Omest} and \eqref{LinfOmest} of Lemma \ref{Wnperp} allow us to estimate 
\beq
\label{L1omest}
\norm{\langle \Om_n^i,\Psi_n\rangle\Om_n^i}_{L^1(M,g_n)}
\leq 
\|\Psi_n\|_{L^1}
\|\Om_n^i\|_{L^\infty}\|\Om_n^i\|_{L^1}
\leq C\norm{\Psi_n}_{L^1(M,g_n)}.
\eeq
Combining, we find that 
$$\norm{P_{g_n}(\Psi_n)}_{L^1(M,g_n)}\leq C\norm{\Psi_n}_{L^1(M,g_n)},$$
which contradicts \eqref{to_contradict} and establishes
\eqref{est:Linfty-Projection-wholesurface}.

Next we turn to proving
\eqref{est:principal-part-lemma} and \eqref{est:decay-part-lemma}, but with the latter initially only required for $\de\in (0,\de_0]$, where $\de_0$ is as in \eqref{RTZLemma2.2}.
Now in order to argue by contradiction, we suppose instead that there exist
an oriented closed surface $M$ of genus $\ga\geq 2$, a sequence of metrics $g_n\in \M$, a sequence of elements 
$\Psi_n \in \Qu_{L^2}(M,g_n)$
and a sequence of collars $\Col_n=\Col(\ell_n)$ in $(M,g_n)$ 
with $\ell_n<2\arsinh(1)$, 
such that 
at least one of the estimates
\beq \label{est-tocontradict1}
\abs{b_0(P_{g_n}(\Psi_n),\Col_n)-\frac{\ell_n^3}{32\pi^5}\langle\Psi_n,dz^2\rangle_{L^2(\Col_n)}}\leq n\cdot \ell_n^3\norm{\Psi_n}_{L^1(M,g_n)}\eeq
or 
\beq \label{est-tocontradict2}
\norm{\om^\perp(P_{g_n}(\Psi_n),\Col_n) }_{L^\infty(\delta\thin(\Col_n))}\leq
n\norm{ \Psi_n}_{L^1(M,g_n)}\cdot \delta^{-2}e^{-\pi/\delta} \text{ for every } \delta\in (0,\de_0]\eeq 
is violated for each $n$. We will show that in fact, both are satisfied, for a subsequence,  even with the coefficients $n$ replaced with a large constant $C$.

As above, we appeal to Lemmata \ref{W_structure} and \ref{Wnperp} to give a subsequence and a decomposition as in \eqref{eq:PnPsin}, with $k$ (possibly zero) degenerating collars identified.
By passing to a further subsequence, we may further assume that either $\ell_n\to 0$ or $\ell_n$ has a positive lower bound, uniform in $n$.

Let us deal first with the harder case that $\ell_n\to 0$,
in which case we may assume that $\Col_n$ corresponds to the first of the $k$ degenerating collars, with corresponding basis element
$\Om_n^1\in W_n^\perp$.
The essential idea is that out of the $k+1$ terms in the 
decomposition \eqref{eq:PnPsin}, only the $\Om_n^1$ term will contribute substantially to the restriction of 
$P_{g_n}(\Psi_n)$ to the thin part of the collar $\Col_n$. 

Let us consider $w_n$ first.
Since it has vanishing principal part on each degenerating collar,
i.e. $w_n=\om^\perp(w_n)$, we can 
apply \eqref{est-collar-decay} and 
estimate as in \eqref{est:wn_est} to give
\beqa \label{est:col-decay-1}
\norm{w_n}_{L^\infty(\delta\thin(\Col_n))}&=\norm{ \om^\perp(w_n)}_{L^\infty(\delta\thin(\Col_n))}
\leq C\delta^{-2}e^{-\pi/\delta}\norm{w_n}_{L^1(M)}\\
&\leq C\delta^{-2}e^{-\pi/\delta}\norm{\Psi_n}_{L^1(M)} \text{ for all } \de\in (0,\de_0],
\eeqa
with $C$ independent of $n$. 
Here and in the following all norms are computed with respect to $g_n$ and we abbreviate
 $b_0(\cdot)=b_0(\cdot,\Col_n)$ and
$\om^\perp(\cdot)=\om^\perp(\cdot,\Col_n)$. 

To analyse $\langle\Psi_n,\Om_n^i\rangle\Om_n^i$ 
we first use \eqref{LinfOmest} to bound
\beq \abs{\langle\Psi_n,\Om_n^i\rangle_{L^2(M)}}\leq C\cdot (\ell_n^i)^{-1/2}\norm{\Psi_n}_{L^1(M)}.\label{est:innerproduct} \eeq
Recall that the collars $(\Col_n^i)_{i=1}^k$ are disjoint, so using \eqref{est:Om-complement} and the orthogonality of principal and collar decay part on subcollars,
we obtain that for  $i\neq 1$ 
$$ \norm{\om^\perp(\Om_n^i)}_{L^2(\de_0\thick(\Col_n))}\leq \norm{\Om_n^i}_{L^2(\de_0\thick(\Col_n))}\leq \norm{\Om_n^i}_{L^2(M\setminus \Col_n^i)}
\leq C(\ell_n^i)^{3/2},$$
(with $\de_0$ still that from \eqref{RTZLemma2.2})
which, combined with \eqref{RTZLemma2.2} and \eqref{est:innerproduct}, gives that for every $i\neq1$ and $\de\in (0,\de_0]$
\beq\label{est:col-decay-2}
\norm{\om^\perp(\langle\Psi_n,\Om_n^i\rangle\Om_n^i)}_{L^\infty(\de\thin(\Col_n))}
\leq 
C\ell_n^i\delta^{-2}e^{-\pi/\delta}\norm{\Psi_n}_{L^1(M)}, 
\eeq
again with $C$ independent of $n$. 
On the other hand, for $i=1$, a combination of \eqref{RTZLemma2.2} and \eqref{est:Om-collar} allows us to estimate
\begin{equation*}
\begin{aligned}
\|\om^\perp(\Om^1_n)\|_{L^\infty(\de\thin(\Col_n))}
&\leq
C\de^{-2}e^{-\pi/\de}
\|\om^\perp(\Om^1_n)\|_{L^2(\de_0\thick(\Col_n))}\\
&\leq
C\de^{-2}e^{-\pi/\de}
\|\om^\perp(\Om^1_n)\|_{L^\infty(\de_0\thick(\Col_n))}\\
&\leq
C\de^{-2}e^{-\pi/\de}\ell_n^{3/2}
\end{aligned}
\end{equation*}
for $\de\in (0,\de_0]$.
We can combine this with \eqref{est:innerproduct} 
to find that the collar decay part is small:
\begin{equation*}
\norm{\om^\perp(\langle\Psi_n,\Om_n^1\rangle\Om_n^1)}_{L^\infty(\de\thin(\Col_n))}
\leq C\ell_n\delta^{-2}e^{-\pi/\delta}\norm{\Psi_n}_{L^1(M)} \text{ for all } \de\in (0,\de_0],
\end{equation*}
$C$ independent of $n$, which in view of  \eqref{est:col-decay-1} and \eqref{est:col-decay-2} means that \eqref{est-tocontradict2} is fulfilled for all $n$ sufficiently large. 

We then note that estimate \eqref{est:princ_different_collar} implies that also the contribution of the $\Om_n^i$ 
to the principal part of 
$P_{g_n}(\Psi_n)$
on $\Col_n$ is small if $i\neq 1$, namely using once more \eqref{est:innerproduct}, we get
\beq\label{est:princ-1}
\abs{b_0(\langle\Psi_n,\Om_n^i\rangle\Om_n^i)}\leq C\ell_n^i (\ell_n)^{3}\norm{\Psi_n}_{L^1(M)}\leq C(\ell_n)^{3}\norm{\Psi_n}_{L^1(M)}.
\eeq

In order to evaluate the principal part of the dominating term $\langle\Psi_n,\Om_n^1\rangle\Om_n^1$, we first use 
\eqref{est:Om-complement} and \eqref{est:Om-collar} from Lemma \ref{Wnperp}, and abbreviate $\al:=1/(32\pi^5)^{1/2}$
to estimate
\beqa
\label{peach2}
\bigg|\langle \Psi_n,\Om_n^1 \rangle_{L^2(M)}-
\al\ell_n^{3/2}\langle \Psi_n,dz^2 \rangle_{L^2(\Col_n)}\bigg|
&\leq \abs{\langle \Psi_n,\Om_n^1 \rangle_{L^2(M\backslash\Col_n)}}
\\ & \quad +
\bigg|\langle \Psi_n,\Om_n^1-b_0(\Om_n^1)dz^2\rangle_{L^2(\Col_n)}\bigg|\\
&\quad +\abs{b_0(\Om_n^1)-\al\ell_n^{3/2}}\cdot \abs{\langle \Psi_n, dz^2 \rangle_{L^2(\Col_n)}}
\\
&\leq C\ell_n^{3/2}\|\Psi_n\|_{L^1(M)},
\eeqa
using \eqref{sizes_on_collars}.

Combined with 
\eqref{est:Om-collar} and \eqref{est:innerproduct}, 
estimate \eqref{peach2} implies that 
\beqa \label{est:tobeimprovedfordz2}
\lefteqn{\abs{b_0(\langle \Psi_n,\Om_n^1 \rangle\Om_n^1)-\al^2\ell_n^{3}\langle \Psi_n,dz^2 \rangle_{L^2(\Col_n)}}}\qquad\\
&\leq
\left|\langle \Psi_n,\Om_n^1 \rangle (b_0(\Om_n^1)-\al\ell_n^{3/2})\right|+\left|\al\ell_n^{3/2}\langle \Psi_n,\Om_n^1 \rangle 
-\al^2\ell_n^{3}\langle \Psi_n,dz^2 \rangle_{L^2(\Col_n)}\right|\\
&\leq C\ell_n^{-1/2}\|\Psi_n\|_{L^1(M)}
\ell_n^{9/2}+\al\ell_n^{3/2} C\ell_n^{3/2}\|\Psi\|_{L^1(M)}\\
&\leq C\ell_n^3 \norm{\Psi_n}_{L^1(M)}.
\eeqa
Since any other contribution to the principal part of $P_{g_n}\Psi_n$ is bounded 
by \eqref{est:princ-1} this implies that also \eqref{est-tocontradict1} is fulfilled for all sufficiently large $n$, leading to a contradiction to our assumption in the case $\ell_n\to 0$.

Next we need to deal with the easier case that $\ell_n$ has a positive lower bound, independent of $n$, and thus the injectivity radius on $(\Col_n,g_n)$ has a positive lower bound.
In this case, when we decompose $P_{g_n}(\Psi_n)$ as in 
\eqref{eq:PnPsin}, the collar $\Col_n$ will be disjoint from the $k$ degenerating collars. 
We can argue just as above in order to establish
\eqref{est:col-decay-1} for $\om^\perp(w_n)$ and \eqref{est:col-decay-2}, but this time the latter holds also for $i=1$.
In this simpler case, those two estimates are already enough to establish \eqref{est-tocontradict2} for sufficiently large $n$, by arguing as above.

In this case that $\ell_n$ (and thus the injectivity radius on $\Col_n$) has a positive lower bound, establishing \eqref{est-tocontradict1} is simply a matter of estimating the two terms on the left-hand side \emph{individually} -- it is not just the difference that is controlled.
To estimate $b_0:=b_0(P_{g_n}(\Psi_n),\Col_n)$, we note that
$$\|b_0 dz^2\|_{L^2(\Col_n)}\leq
\|P_{g_n}(\Psi_n)\|_{L^2(\Col_n)}\leq
C\|P_{g_n}(\Psi_n)\|_{L^1(M)}
\leq
C\|\Psi_n\|_{L^1(M)},$$
by \eqref{est:holo-thick} and the first part \eqref{est:Linfty-Projection-wholesurface} of the proposition.
But it is easy to see that $\|dz^2\|_{L^2(\Col_n)}$ has a uniform lower bound -- for example, by Cauchy-Schwarz, it can be controlled from below in terms of 
the (bounded) area of $(\Col_n,g_n)$ and 
$$\|dz^2\|_{L^1(\Col_n)}
=8\pi X(\ell_n)\geq \frac{2\pi^3}{\arsinh(1)}
$$
(see 
\eqref{X_lower_bd}
and \eqref{sizes_on_collars}). 
Therefore we have
\beq
\label{b0est}
|b_0|\leq C\|b_0 dz^2\|_{L^2(\Col_n)}\leq
C\|\Psi_n\|_{L^1(M)}.
\eeq
Meanwhile we can directly estimate the other term of 
\eqref{est-tocontradict1} by
$$\bigg|\frac{\ell_n^3}{32\pi^5}\langle\Psi_n,dz^2\rangle_{L^2(\Col_n)}\bigg|
\leq 
C \ell_n^3\|\Psi_n\|_{L^1(M)}
\|dz^2\|_{L^\infty(\Col_n)},$$
and by \eqref{sizes_on_collars} (and the boundedness of $\ell_n$)
we deduce
$$\bigg|\frac{\ell_n^3}{32\pi^5}\langle \Psi_n,dz^2\rangle_{L^2(\Col_n)}\bigg|
\leq 
C \|\Psi_n\|_{L^1(M)}.$$
Combining with \eqref{b0est}, and keeping in mind the uniform positive lower bound for $\ell_n$, we deduce that 
\eqref{est-tocontradict1} holds for sufficiently large $n$.

At this point, we have succeeded in proving 
\eqref{est:principal-part-lemma} and \eqref{est:decay-part-lemma}, but with the latter only required for $\de\in (0,\de_0]$.
To establish the same claim for the full range 
$\delta\in(0,\arsinh(\cosh(\ell/2)))$,
it suffices to observe additionally that by \eqref{dethickest} and by \eqref{est:Linfty-Projection-wholesurface} 
$$
\norm{P_g(\Psi)-b_0dz^2}_{L^\infty(\de_0\thick(\Col,g))}\leq 
C\norm{P_g(\Psi)}_{L^1(M,g)}\leq 
C\norm{ \Psi}_{L^1(M,g)}.
$$
\end{proof}

\subsection{Proof of the general formula for $\frac{d\ell}{dt}$, Lemma \ref{lengthevol}}
\label{are_we_done}
We can now prove Lemma \ref{lengthevol} based on the formula for the projection derived in the previous section. 

In keeping with Lemma \ref{lengthevol}, throughout this section we assume that 
$M$ is an oriented closed surface of genus $\gamma\geq 2$ and $g(t)$ is a smooth 
one-parameter family of metrics in $\M$ such that 
$\partial_tg\vert_{t=0}=Re(P_g(\Psi))$ for some $\Psi\in \Qu_{L^2}(M,g(0))$,
and we assume that $(M,g(0))$ contains a collar $ \Col$ around a simple closed geodesic $\si$ of length $\ell<2\arsinh(1)$.
\brmk
\label{regularity_rmk}
As $t$ varies near $0$, the locally minimising geodesic $\si$ will vary as a one-parameter family of simple closed geodesics $\si(t)$ of length $\ell(t)$ with respect to $g(t)$. This family will be continuous with respect to (say) the $C^1$ topology. As one would expect, given that $\si$ is a geodesic, we claim that
$$\frac{d}{dt}\bigg|_{t=0}\ell(t)=\frac{d}{dt}\bigg|_{t=0}L_{g(t)}(\si).$$
To see this, note that for each $s$, the function
$$t\mapsto L_{g(t)}(\si(s))$$
is smooth, and lies above the function $t\mapsto \ell(t):=L_{g(t)}(\si(t))$ 
(with equality at $t=s$)
because $\si(t)$ minimises the length in $(M,g(t))$ over all nearby simple closed curves.
But by the continuity of $\si(t)$ in $C^1$, we see that 
$$s\mapsto \frac{d}{dt}\bigg|_{t=s}L_{g(t)}(\si(s))$$
is continuous, 
which is enough to conclude.
\ermk

\begin{proof}[Proof of Lemma \ref{lengthevol}]
At $t=0$, writing $g_{\th\th}=g\left(\pl{}{\th},\pl{}{\th}\right)$, we have 
$$\ell=\int_{\si}g_{\th\th}^{1/2}d\th=2\pi g_{\th\th}^{1/2},$$
i.e. $g_{\th\th}=\left(\frac{\ell}{2\pi}\right)^2$.
By Remark \ref{regularity_rmk}, we have at $t=0$
\beqas
\label{kangaroo}
\frac{d\ell}{dt}&=\frac{d}{dt}L_{g(t)}(\si)\\
&=\int_{\si}\left(\half \left[g_{\th\th}\right]^{-1/2}\pl{g_{\th\th}}{t}\right)d\th=\frac{\pi}{\ell}\int_{\si}\pl{g_{\th\th}}{t}d\th
=\frac{\pi}{\ell}\int_{\si}Re(P_g(\Psi))
{\textstyle \left(\pl{}{\th},\pl{}{\th}\right)}
d\th\\
&=\frac{\pi}{\ell}\int_{\si}Re(b_0(P_g(\Psi))dz^2){\textstyle \left(\pl{}{\th},\pl{}{\th}\right)}+Re(\om^\perp(P_g(\Psi))){\textstyle \left(\pl{}{\th},\pl{}{\th}\right)} d\theta\\
&=-\frac{2\pi^2}{\ell}Re(b_0(P_g(\Psi))),
\eeqas
where we split $P_g(\Psi)=b_0(P_g(\Psi))dz^2+\om^\perp(P_g(\Psi))$ into its principal part
and its collar decay part, and notice that the latter integrates to zero by \eqref{Fourier_decomp}.
 
Proposition \ref{lemma:projection} tells us that
$$\bigg|Re(b_0(P_g(\Psi)))-\frac{\ell^3}{32\pi^5}Re\langle\Psi, dz^2\rangle_{L^2(\Col)}\bigg|\leq C\ell^3\norm{\Psi}_{L^1},$$
so indeed, by \eqref{kangaroo},
$$\bigg|\frac{d\ell}{dt}+\frac{\ell^2}{16\pi^3}Re\langle\Psi, dz^2\rangle_{L^2(\Col)}\bigg|\leq C(\ell^{-1}\ell^3 
)\norm{\Psi}_{L^1}\leq  C\ell^2\norm{\Psi}_{L^1}.$$
\end{proof}

\begin{rmk}
\label{simple_formula_holo_case}
From the proof we immediately see (as is already well-known, e.g. \cite[Theorem 3.3]{Wolpert07}) that for \textit{holomorphic} quadratic differentials $\Psi$, we have
\beq \label{est:ddtl_holomorphic} \frac{d\ell}{dt}=-\frac{2\pi^2}{\ell}Re(b_0(\Psi)).
\eeq
\end{rmk}

\subsection{Incompleteness of Teichm\"uller space}
\label{TeichIncomp}

In order to illustrate the use of Lemma \ref{lengthevol}, we show the well-known fact (e.g. Wolpert \cite{wolpert75}) that Teichm\"uller space equipped with the Weil-Petersson metric is incomplete. Indeed, if we pick a metric $g$ on any oriented closed surface $M$ of genus at least $2$, with a collar $\Col$ having 
$\ell<2\arsinh(1)$, then we may deform it as in Lemma \ref{lengthevol} taking $\Psi$ to be $dz^2$ on $\Col$ and zero elsewhere. In this case, the distance $s(t)$ travelled through Teichm\"uller space is, by definition (with one choice of normalisation of the Weil-Petersson metric)
$$\frac{ds}{dt}=\frac14\|P_g(\Psi)\|_{L^2(M)}\leq \frac14\|\Psi\|_{L^2(M)}
= \frac14\|dz^2\|_{L^2(\Col)}=\left(\frac{2\pi^5}{\ell^3}\right)^{1/2}
(1+O(\ell^3)),$$
as a result of \eqref{sizes_on_collars}.
Meanwhile, by Lemma \ref{lengthevol}, we have

$$\bigg|\frac{d\ell}{dt}+\frac{\ell^2}{16\pi^3}
\|dz^2\|^2_{L^2(\Col)}\bigg|
\leq C\ell^2\|dz^2\|_{L^1(\Col)},$$
and hence, by \eqref{sizes_on_collars} 
$$\bigg|\frac{d\ell}{dt}
+\frac{2\pi^2}{\ell}
\bigg|
\leq C\ell.$$
Combining these facts, we find that
$$\frac{d\ell^{1/2}}{ds}\leq-\left(\frac{1}{2\pi}\right)^{1/2}
+O(\ell^2),$$
and thus we can pinch a collar by moving a distance no more
than $(2\pi\ell)^{1/2}+O(\ell^{5/2})$ in Teichm\"uller space. 
We do not claim this to be optimal in any way; already our results
would allow us to show a stronger upper bound of $(2\pi\ell)^{1/2}+O(\ell^{7/2})$, but indeed it was proven 
by Wolpert in \cite{Wolpert07} that $\ell^\half$ is convex and consequently this distance is bounded from above by $(2\pi\ell)^{1/2}$ itself, with a lower bound of  
$dist\geq (2 \pi\ell)^{1/2}+O(\ell^{5/2})$ established in the same paper.

Remark that combining the (essentially explicit) upper bound 
\eqref{est:apriori-upper-bound-b0}
on the principal part of any unit holomorphic quadratic differential with \eqref{est:ddtl_holomorphic}  allows us to improve this lower bound to an estimate of the form
\beq \label{est:lower}
\dist\geq (2\pi\ell)^{1/2}\left(1-\frac1{84\pi}\ell^3+O(\ell^5)\right).
\eeq

\newcommand{\J}{{\mathcal{J}}}
In the more general case that a collection $\{\si_i\}_{i\in\J}$
of geodesics pinches, i.e. in which one considers the distance to the 
part (or stratum) of the boundary characterised by 
$\ell :=\sum_{i\in\J}\ell_i=0$, the lower bound in the estimate 
$(2\pi\ell)^{1/2}+O(\ell^{5/2})\leq dist \leq (2\pi\ell)^{1/2}$
proven in \cite{Wolpert07} can be improved to 
\eqref{est:lower}
by a similar argument, using additionally that the degenerating collars are disjoint.

\section{Controlling the weighted energy $I$}\label{sec:I}
In this section, 
we finally prove the estimate on the full weighted energy 
$$I=\int_\Col e(u,g)\rho^{-2}dv_g$$ 
that we claimed in Lemma \ref{lemma:Icontrol}. 
Let $(u,g)$ be any solution of the Teichm\"uller harmonic map flow \eqref{flow} defined on an interval $[0,T)$ and let $t_0\in[0,T)$ be such that $(M,g(t_0))$ contains a collar $\Col_{t_0}$ around 
a simple closed geodesic $\si(t_0)$ of length 
$\ell(t_0)<2\arsinh(1)$. 
As in Section \ref{are_we_done}, for $t$ close to $t_0$, this geodesic will vary continuously 
through a family of simple closed geodesics $\si(t)$
in $(M,g(t))$, each of which is at the centre of a collar $\Col_t$.
Every closed subset of $\Col_{t_0}$
will also be contained in  $\Col_{t}$ for $t$ sufficiently close to $t_0$.

We may thus consider the evolution of the associated weighted energies $I$, or rather of a smoothed-out version of $I$ given by 
\beq \label{I-zeppelin} 
\I=\I(u(t),g(t))=\int_{\Col_t} e(u(t),g(t))\rho^{-2}(t)\vph^2(\rho(t))dv_{g(t)},\eeq
$\vph\in C_0^\infty([0,2\de),[0,1])$ a cut-off function with $\vph\equiv 1$ on $[0,\de]$ and 
$|\vph'|\leq 2/\delta$, where we require $\de>0$ small enough such that 
$2\de \leq \rho(X(\ell))$ for all $\ell\in(0,2\arsinh(1))$.
Indeed, by \eqref{rho_range}, we can fix $\de=\frac{1}{2\pi}$, which relieves any dependencies of constants on $\de$.
Note that $I$ and $\I$ are related in the sense that
\beq
\label{I_I_est}
0\leq I-\I\leq \de^{-2}E_0,
\eeq
where $E_0$ is an upper bound on the total energy.

The main step in the proof of Lemma \ref{lemma:Icontrol} is to show
\begin{lemma}\label{lemma:I-zeppelin}
 Let $(u,g)$ be a solution of \eqref{flow} on an oriented closed 
 surface $M$ of genus at least 2, for $t\in [0,T)$, into a target $N$ that supports no bubbles.
Then at any time $t\in [0,T)$ at which $(M,g(t))$ contains a collar $\Col=\Col(\ell)$ with 
$\ell<2\arsinh(1)$, the corresponding weighted energy $\I$ defined by 
 \eqref{I-zeppelin} satisfies
 \beq \label{est:I-zeppelin}
 \bigg|\ddt \log(1+\I)\bigg| \leq C\left(1+\norm{\tau_g(u)}_{L^2(M,g)}^2\right)\eeq
 for a constant $C$ depending only on $M$, $N$, $\eta$ and an upper bound $E_0$ on the initial energy.
\end{lemma}
Accepting this lemma for the moment, we can finally give the:
\begin{proof}[Proof of Lemma  \ref{lemma:Icontrol}]
Given $(u,g)$ as in Lemma \ref{lemma:Icontrol} and a time $t_0\in [0,T)$ such that $(M,g(t_0))$ contains a collar around a geodesic $\si(t_0)$
of length $L_{g(t_0)}(\si(t_0))<
2\arsinh(1)
$ we let $t_{min}\geq 0$ be the minimal number
such that there is a continuous family of simple closed geodesics $(\si(t))_{t\in[t_{min},t_0]}$ in $(M,g(t))$, as in Section \ref{are_we_done}, 
with 
$L_{g(t)}(\si(t))<2\arsinh(1)
$ for all $t\in(t_{min},t_0]$.

If $t_{min}=0$, we can initially bound the weighted energy $\I$ in terms of $E(u(0),g(0))\leq E_0$ and $\ell_0=L_{g(0)}(\si(0))\geq 2\inj_{g(0)}M$ as
$$\I(0)\leq \big(\sup_{\Col}\rho^{-2}\big)\int_{\Col} e(u,g)dv_g\bigg\vert_{t=0}\leq \left(\frac{2\pi}{\ell_0}\right)^2E(u(0),g(0))\leq C \cdot(\inj_{g(0)}M)^{-2},$$
with $C=C(E_0)$. Since the space-time integral of the squared tension is bounded by the initial energy, by \eqref{energy-identity}, 
integration of \eqref{est:I-zeppelin} from $t=0$ to $t_0$ gives the desired estimate
\beq 
\I(t_0)\leq \exp\left[C\int_0^{t_0}\left(1+\norm{\tau_g(u)}_{L^2(M,g)}^2\right)dt\right] \cdot (\I(0)+1) \leq e^{C(t_0+1)}\cdot (1+(\inj_{g(0)}M)^{-2})\eeq
first for $\I$, and then, by \eqref{I_I_est}, 
also for the original weighted energy $I$.

On the other hand, if $t_{min}>0$ then $L_{g(t_{min})}(\si(t_{min}))=
2\arsinh(1)
$ and thus $\I(t_{min})\leq C
E_0$ so integration of \eqref{est:I-zeppelin} from 
$t_{min}$ to $t_0$ again proves Lemma \ref{lemma:Icontrol}. 
\end{proof}

We now turn to the proof of Lemma \ref{lemma:I-zeppelin}. 
\newcommand{\ddeps}{\left.\frac{d}{d\eps}\right\vert_{\eps=0}}
To begin with, we derive a formula for the evolution of the 
conformal factor $\rho$ along certain curves of hyperbolic metrics.
%
We recall, by \eqref{eq:rho-inj}, that the conformal factor $\rho$ on a collar $\Col=\Col(\ell(t))\subset (M,g(t))$ 
can be characterised in a coordinate-free way as 
\beq \label{eq:rho} 
\rho(p,t)=\frac{\ell(t)}{2\pi\sinh(\ell(t)/2)}\cdot \sinh\left(\inj_{g(t)}(p)\right). \eeq

\begin{lemma}\label{lemma:rho}
Let $(g(t))_{t\in[0,T)}$ be a smooth curve of hyperbolic metrics on an oriented closed surface $M$ such that
$$\pt g=Re(P_g(\Psi(t))) \text{ for some } \Psi(t)\in \Qu_{L^2}(M,g(t)),$$
and assume that at some time $t_0\in [0,T)$, the surface $(M,g(t_0))$ contains a collar $\Col$ around a geodesic of length 
$\ell<2\arsinh(1)$.
Then $\rho(p,t)=\rho(\ell(t),\inj_{g(t)}(p))$ evolves according to 
\beqs 
\label{target}
\abs{\pt (\rho^2)(p)+Re(b_0)}\leq C\cdot e^{-1/\rho(p)}\norm{\Psi}_{L^1(M,g)}\quad\text{ for all } p\in \Col
\eeqs
at time $t_0$, where $b_0dz^2=b_0(P_g(\Psi),\Col)dz^2$ is the principal part 
on $\Col$, and $C<\infty$
depends only on the genus of $M$.
\end{lemma}

\begin{proof}
In the lemma, the metric $g_0:=g(t_0)$ is being deformed in the direction $\pt g=Re(b_0 dz^2)+Re(\om^\perp)$, where 
$\om^\perp=\om^\perp(P_{g_0}(\Psi(t_0)),\Col)$.
Heuristically, it is the first of these terms that is dominant. Indeed, if we consider an alternative, smooth \emph{symmetric} flow of hyperbolic metrics $\hat g(t)$ on $\Col$ for $t$ near $t_0$, with $\hat g(t_0)=g_0$ and $\pt \hat g=Re(b_0 dz^2)$ (one could write down such a flow explicitly) then the corresponding conformal factor $\hat \rho$ could be written 
at $q=(s_0,\theta_0)\in \Col$
independently of the time-$t$ collar coordinates as
$$L_{\hat g(t)}(\{s_0\}\times S^1)=2\pi\hat\rho(q,t),$$
because of the symmetry of the deformation.
In particular, we would have at $t=t_0$ that
\beqa  \label{hat_comp}
\pt(\hat\rho^2)(q)&=2\rho(s_0)\cdot \frac1{2\pi}\ddt L_{\hat g(t)}(\{s_0\}\times S^1)\\
&= \frac{\rho(s_0)}{2\pi} \int_{\{s_0\}\times S^1}\big((g_0)_{\th\th}\big)^{-1/2}\pt \hat g_{\th\th} d\th\\
&=-Re(b_0).
\eeqa
Another way of computing the derivative of the conformal factor $\hat\rho$, or indeed $\rho$, is via \eqref{eq:rho}.
Writing $F(x)=\frac{x}{2\pi\sinh(x/2)}$ so that 
$\rho(q,t)=F(\ell(t))\sinh(\inj_{g(t)}(q))$,
we compute at $t=t_0$
$$\pt \rho(q)=F'(\ell)\frac{d\ell}{dt}\sinh(\inj_{g_0}(q))
+F(\ell)\cosh(\inj_{g_0}(q))\pt (\inj_{g(t)}(q)).$$
In order to compute $\pt (\inj_{g(t)}(q))$, we note that $\iota:=\inj_{g_0}(q)$ can be realised as half the length of a unit speed geodesic $\si:[0,2\iota]\to \Col_\iota$ mapping the end points to $q$ and wrapping once around the collar, where 
$\Col_\iota:= \{p\in\Col\ :\ \inj_{g_0}(p)\leq\iota\}$
is the closure of $\iota\thin(\Col,g_0)$ when this thin part is nonempty.
More generally, for $t$ close to $t_0$, 
$$\inj_{g(t)}(q)=\half L_{g(t)}(\si(t))$$
for an appropriate continuous family of geodesics $\si(t)$ 
in $(M,g(t))$ with $\si(t_0)=\si$ and 
with fixed end points.
Adapting the argument of Remark 
\ref{regularity_rmk} gives
$$\pt\inj_{g(t)}(q)=\half \pt L_{g(t)}(\si).$$
But we can compute
$$\pt L_{g(t)}(\si)=\half\int_0^{2\iota}\pt g(\dot\si,\dot\si),$$
and so assembling what we have seen, we get at $t=t_0$ that
\beq
\label{key_rho_squared}
\pt \rho^2(q)=2\rho(q)F'(\ell)\frac{d\ell}{dt}\sinh(\inj_{g_0}(q))
+\half\rho(q)F(\ell)\cosh(\inj_{g_0}(q))
\int_0^{2\iota}\pt g(\dot\si,\dot\si).
\eeq
This formula equally well applies to the flow $\hat g(t)$, and so noting that $\frac{d\ell}{dt}=\frac{d\hat\ell}{dt}$ at $t=t_0$ 
(because the collar decay part $\om^\perp$ does not contribute to $\frac{d\ell}{dt}$) we obtain from \eqref{hat_comp} that 
\beq
-Re(b_0)=2\rho(q)F'(\ell)\frac{d\ell}{dt}\sinh(\inj_{g_0}(q))
+\half\rho(q)F(\ell)\cosh(\inj_{g_0}(q))
\int_0^{2\iota}Re(b_0 dz^2)(\dot\si,\dot\si).
\eeq
This allows us to simplify \eqref{key_rho_squared} when applied to $g(t)$, to
\beq
\label{simp_rho_squared}
\pt \rho^2(q)=-Re(b_0)+\half\rho(q)F(\ell)\cosh(\inj_{g_0}(q))
\int_0^{2\iota}Re(\om^\perp)(\dot\si,\dot\si),
\eeq
and in particular, 
by \eqref{est:decay-part-lemma} of Proposition \ref{lemma:projection}, 
we find that 
\beqa 
\big|\pt \rho^2(q)+Re(b_0)\big|
&\leq C\rho(q)\iota\|\om^\perp\|_{L^\infty(\Col_\iota
)}\\
&\leq C\rho(q)\iota^{-1}e^{-\pi/\iota}\|\Psi\|_{L^1(M,g_0)}\\
&\leq Ce^{-1/\rho(q)}\|\Psi\|_{L^1(M,g_0)}
\eeqa
as desired, because $x\mapsto x^{-1}e^{-\pi/x}$ is increasing for $x\in (0,\pi)$, and $\iota\leq \pi\rho(q)$ by \eqref{est:rho-inj-appendix}.
\end{proof}

To apply this lemma to solutions of the Teichm\"uller harmonic map flow we observe:

\begin{rmk}
It is a consequence of
Proposition \ref{lemma:projection}, \eqref{sizes_on_collars} and the definition $\Phi(u,g)=(|u_s|^2-|u_\th|^2-2i\langle u_s,u_\th\rangle)dz^2$, 
that the principal part $b_0dz^2$ of $\frac{\eta^2}{4}P_g(\Phi(u,g))$ on a collar
is given by the weighted integrals 
\beqas
Re(b_0)&=\frac{\ell^{3}}{32\pi^5}\eta^2\int_{\Col} (\abs{u_s}^2-\abs{u_\theta}^2)\rho^{-4}dv_g+r_1\\
Im(b_0)&=-\frac{\ell^{3}}{16\pi^5}\eta^2\int_{\Col} \langle u_s, u_\theta\rangle \rho^{-4}dv_g+r_2,
\eeqas
with error terms $r_1$, $r_2$ bounded by 
$$\abs{r_1}+\abs{r_2}\leq C\ell^3\cdot\norm{\Phi }_{L^1}\leq C\ell^3E_0,$$ 
$C$ depending only on $\gamma$ and $\eta$. 
In particular, 
\beq \label{est:a} 
\abs{Re(b_0)}\leq C \ell^3\int_\Col e(u,g)\rho^{-2}\vph dv_g+C\ell^3\leq C\ell^3(\I+1)
\eeq
while 
\beq \label{est:b}
\abs{Im(b_0)}\leq C\ell^3 \left((I^{(\theta)} \I )^{1/2}
+1\right),
\eeq 
contains the weighted angular energy
$I^{(\theta)}$ controlled by Lemma \ref{ang_en_est} ($C$ now also depending on $E_0$).
\end{rmk}

We can now finally estimate the evolution of weighted energy $\I$ defined in \eqref{I-zeppelin}.

\begin{proof}[Proof of Lemma \ref{lemma:I-zeppelin}]
Let $(u,g)$ be a solution of \eqref{flow} as in Lemma \ref{lemma:I-zeppelin}.
%
The first equation of \eqref{flow} can be written as
$$\pt u-\Delta_{g} u=A_{g}(u)(\na u,\na u)\perp T_uN,$$
which can then be multiplied by $\rho^{-2}\vph^2 \pt u\in T_uN$ (where $\vph$ and its derivative will always be evaluated at $\rho(p,t)$) and integrated over the collar to obtain
\beqas 
0&=\int \big[\abs{\pt u}^2-\pt u\Delta_g u\big] \stw dv_g =\int\abs{\pt u}^2\stw dv_g+\int \langle du,d ( \pt u\rho^{-2}\vph^2)\rangle_{g} dv_g\\
&=\int \abs{\pt u}^2\stw dv_g+\int\langle du,\pt du\rangle_g \stw dv_g +\int \lan du,d(\rho^{-2})\ran_g\pt u\vph^2 dv_g \\
&\quad +\int  \lan du,d (\vph^2\circ \rho) \ran\pt u \rho^{-2}dv_g.
\eeqas
Thus 
\beqa\label{est:I-neu-2}
&\int\abs{\pt u}^2\stw dv_g+\ddt \I(u(t),g(t))
\\ & \hspace{1cm}
\leq 
\ddeps \I(u(t),g(t+\eps))-\int \lan du,d(\rho^{-2})\ran_g\pt u\vph^2 dv_g\\
 & \hspace{1cm}\qquad -2\int \vph\vph' \lan du,d \rho\ran_g\pt u\rho^{-2}dv_g,\eeqa
where $\I(u,g)$ is given by \eqref{I-zeppelin}.  
As $\abs{d\rho}_{g}=\rho^{-1}\abs{\rho'}\leq \rho$, see \eqref{eq:rho_deriv},
we can estimate
$$-2\int \vph\vph' \lan du,d \rho\ran_g\pt u \rho^{-2}dv_g\leq 
C\int \abs{du}_g\abs{\pt u}\rho^{-1}\vph dv_g
\leq  \tfrac12 \int\abs{\pt u}^2\stw dv_g+CE(u,g),$$
as well as
\beqa 
-\int \lan du,d(\rho^{-2})\ran_g\pt u\vph^2 dv_g\leq 2\int \abs{du}_g\abs{\pt u} \stw dv_g\leq \tfrac12 \int\abs{\pt u}^2\stw dv_g+C \I.
\eeqa

As the energy is uniformly bounded, \eqref{est:I-neu-2} thus reduces to 
\beq \label{est:I-neu-1}
\ddt \I(u(t),g(t))\leq \ddeps \I(u(t),g(t+\eps))
+C(\I+1). \eeq

To estimate the first term on the right-hand side, we rewrite it in collar coordinates $(s,\theta)$ of $(\Col_t,g(t))$, $t$ fixed, and use that $\pt g$ is trace-free (which fixes the volume form) 
and that $g^{s\th}=0$ at time $t$, 
to get
\beqa \label{est:split-dI}
 \ddeps \I(u(t),g(t+\eps))&= \half\int \pt\big(g^{ss}\rho^{-2}\big) \abs{u_s}^2 \vph^2 dv_g
+\int \pt \big(g^{s\theta}\big)\lan u_s,u_\theta \ran \rho^{-2}\vph^2 dv_g\\
& + \half\int \pt\big(g^{\theta\theta}\rho^{-2}\big) \abs{u_\theta}^2 \vph^2 dv_g +\int \abs{du}_g^2\rho^{-2} \pt (\vph\circ  \rho) \vph dv_g\\
& =: T_1+T_2+T_3+T_4.
\eeqa
We remark that the two integrals in 
\[T_1=\half\int \pt\big(g^{ss}\big)\rho^{-2} \abs{u_s}^2 \vph^2 dv_g+\half\int \pt\big(\rho^{-2}\big) g^{ss}\abs{u_s}^2 \vph^2 dv_g\] can be of order $\ell\cdot \I^2$ and thus could not be controlled separately. 
Based on the precise estimates on the evolution of $\rho$ derived in Lemma \ref{lemma:rho} we shall however see that, 
up to an exponentially decaying error, the two integrands 
agree, but appear with opposite signs, and thus cancel. 
Indeed, writing $\pt g=Re(b_0 dz^2)+Re(\om^\perp)$ as usual as the sum of its principal and its collar decay parts and recalling that $\pt (g^{ss})=-\rho^{-4}\pt g_{ss}$, we may use  
Lemma \ref{lemma:rho}, 
to obtain
\beqa 
\abs{g^{ss}\pt\big(\rho^{-2})+\rho^{-2}\pt(g^{ss})}&=\abs{-\rho^{-6}\pt(\rho^2)-\rho^{-6}\pt g_{ss}} 
=\rho^{-6}\abs{\pt(\rho^2)+Re(b_0)+Re(\om^\perp)_{ss}}\\
&\leq 
C\rho^{-6}e^{-1/\rho}\norm{\Phi(u,g)}_{L^1}+\rho^{-4}\abs{\om^\perp}_g.
\eeqa 
As $\om^\perp$ is controlled by \eqref{est:decay-part-lemma} and  as $\norm{\Phi(u,g)}_{L^1}\leq C E(u,g)\leq C$, we can thus estimate
\beqa 
\abs{g^{ss}\pt\big(\rho^{-2})+\rho^{-2}\pt(g^{ss})}&\leq C\rho^{-6}e^{-1/\rho}+C\rho^{-4}\inj_g(p)^{-2}e^{-\pi/\inj_g(p)} 
 \leq 
C\rho^{-6}e^{-1/\rho},
\eeqa where the last inequality is a consequence of $\inj_g(p)\leq \pi\rho$ (see \eqref{est:rho-inj-appendix}) and the fact that $x\mapsto x^{-2}e^{-\pi/x}$ is 
monotone near zero and $ \rho$ is bounded from above. 
Consequently
\beq \label{est:I1}
T_1\leq C\int \rho^{-6} e^{-1/\rho} \abs{u_s}^2 \vph^2 dv_g\leq C E(u,g) \leq C. \eeq
To obtain a bound on $T_2$ we use that 
$\pt(g^{s\theta})
=-\rho^{-4}\pt g_{s\th}
=\rho^{-4}(Im(b_0)-(Re(\om^\perp))_{s\theta})$, with $Im(b_0)$ satisfying
\eqref{est:b} and $\om^\perp$ bounded by \eqref{est:decay-part-lemma}, and estimate
\beqa  \label{est:I2}
T_2&\leq \abs{Im(b_0)}\cdot (\sup_{\Col} \rho^{-2}) \int \abs{u_s}\cdot \abs{u_\theta}\rho^{-4}\vph^2dv_g\\
&\quad +C\int \abs{u_s}\cdot \abs{u_\theta
}\rho^{-4} \inj_g(p)^{-2}e^{-\pi/\inj_g(p)}\vph^2 dv_g\\
&\leq C\ell^{-2}\abs{Im(b_0)}\cdot (I^{(\theta)}\I)^{1/2}+C\int \abs{du}_g^2 \inj_g(p)^{-4}e^{-\pi/\inj_g(p)}dv_g\\
&\leq C\ell \big(I^{(\theta)}  \I+( I^{(\theta)}\I)^{1/2}\big)+C E(u,g)\\
&\leq C \ell I^{(\theta)} \I+C,
\eeqa
where we used  $ \pi \rho\geq \inj_g(p)$ in the second estimate and Young's inequality in the last.
Similarly, combining Lemma \ref{lemma:rho} with \eqref{est:decay-part-lemma} and \eqref{est:a}, we can estimate
\begin{equation*}
\begin{aligned}
\abs{\pt(g^{\theta\theta}\rho^{-2})}&=\abs{\rho^{-6}Re(b_0)-\rho^{-6}(Re(\om^ \perp))_{\theta \theta}-\rho^{-6}\pt(\rho^2)}\\
&\leq 2\rho^{-6}\abs{Re(b_0)}+C\rho^{-4}\inj_g(p)^{-2}e^{-\pi/\inj_g(p)}+C\rho^{-6}e^{-1/\rho}\norm{\Phi(u,g)}_{L^1}\\
&\leq C\ell\rho^{-4}(\I+1)+C\rho^{-2}\big( \inj_g(p)^{-4}e^{-\pi/\inj_g(p)}+\rho^{-4}e^{-1/\rho}\big)\\
&\leq C\ell\rho^{-4}(\I+1)+C\rho^{-2},
\end{aligned}
\end{equation*}
and consequently, noting that
$I^{(\theta)}\leq 2I\leq 2\I+C$ by \eqref{I_I_est}, we have
\beqa \label{est:I3}
T_3&\leq C\ell (\I+1) I^{(\theta)}+C E(u,g)
\leq C\ell (I^{(\theta)}+1)\I +C.
\eeqa

Finally, we recall that $\rho\geq \delta$ on the support of 
$\vph'\circ \rho$ and estimate 
\beqa \label{est:T2} T_4
&= \half \int \abs{du}_g^2\rho^{-3} \vph\vph' \pt (\rho^2)  dv_g \leq C E(u,g)\cdot \sup_{\Col} \abs{\pt (\rho^2)}
\leq C \cdot
\abs{Re(b_0)}+C\\
&\leq C\ell^3(\I+1)+C\leq C,
\eeqa
where we applied Lemma \ref{lemma:rho} in the second, \eqref{est:a} in the third and 
$\I\leq \big(\tfrac{2\pi}{\ell}\big)^2 E(u,g)\leq C\ell^{-2}$ in the last step.

Inserting \eqref{est:I1}-\eqref{est:T2} into  
\eqref{est:split-dI} and combining the resulting estimate with \eqref{est:I-neu-1} 
thus implies
\beq
\ddt \I\leq C(\ell I^{(\theta)}+1)\I +C,
\eeq
which, combined with the angular energy estimate of Lemma \ref{ang_en_est}, yields the desired
bound of
$$\ddt \I\leq C(1+\norm{\tau_g(u)}_{L^2}^2)\cdot (1+\I).$$
This very last estimate is the only place we use the no bubble assumption.
\end{proof}

\appendix
\section{Appendix}
We will need the following `Collar lemma' throughout the paper.

\begin{lemma}[Keen-Randol \cite{randol}] \label{lemma:collar}
Let $(M,g)$ be a closed hyperbolic surface and let $\si$ be a simple closed geodesic of length $\ell$. Then there is a neighbourhood around $\si$, a so-called collar, which is isometric to the 
cylinder 
$$\Col(\ell):=(-X(\ell),X(\ell))\times S^1$$
equipped with the metric $\rho^2(s)(ds^2+d\theta^2)$ where 
$$\rho(s)=\frac{\ell}{2\pi \cos(\frac{\ell s}{2\pi})} 
\qquad\text{ and }\qquad  
X(\ell)=\frac{2\pi}{\ell}\left(\frac\pi2-\arctan\left(\sinh\left(\frac{\ell}{2}\right)\right) \right).$$ 
The geodesic $\si$ then corresponds to the circle $\{(0,\theta)\ |\ \theta\in S^1\}\subset \Col(\ell)$. 
\end{lemma}

In this version of the collar lemma, the intrinsic distance $w$ between the two ends of the collar is related to $\ell$ via
$$\sinh \frac{\ell}{2} \sinh \frac{w}{2}=1,$$
which is sharp.
In order to simplify the discussion of dependency of constants, 
and ensure that different collars do not intersect, we will only
talk about collars with $0<\ell<2\arsinh(1)$ (cf. \cite[Appendix A.2]{RTZ}).
As $X(\ell)$ is decreasing in $\ell$, we then have
\beq
\label{X_lower_bd}
X(\ell)>\frac{\pi^2}{4\arsinh(1)}\qquad\text{ for }0<\ell<2\arsinh(1).
\eeq

We recall (cf. \cite[Lemma A.5]{RTZ}) that 
the injectivity radius is given by the formula
\beq \label{eq:rho-inj} \sinh({\inj}(s,\theta))\cdot \cos\left(\frac{\ell s}{2\pi}\right) =  \sinh\left(\frac{\ell}{2}\right).  \eeq
Note that at the ends of the collar we have
$$\rho(X(\ell))=\rho(-X(\ell))=\frac{\ell}{2\pi \tanh\frac{\ell}{2}}\sim\frac{1}{\pi}\qquad\text{for small }\ell>0,$$
and 
\beq
\label{inj_at_end}
\inj(X(\ell),\th)=\inj(-X(\ell),\th)=\arsinh(\cosh(\frac{\ell}{2}))\sim \arsinh(1)\qquad\text{for small }\ell>0.
\eeq
Within the collar, for $s\in(-X(\ell),X(\ell))$, we have 
\beq
\label{rho_range}
\rho(s)\leq \rho(X(\ell))=\frac{\ell}{2\pi\tanh\frac{\ell}{2}}
\in \left(\frac{1}{\pi},\frac{\sqrt{2}\arsinh(1)}{\pi}\right),
\eeq
for $\ell\in (0,2\arsinh(1))$.
Moreover, we can compute
\beq\label{eq:rho_deriv}
\frac{d}{ds}\log\rho(s)=\frac{\ell}{2\pi}\tan\frac{\ell s}{2\pi},
\qquad
\text{ so }
\qquad
\left|
\frac{d}{ds}\log\rho(s)
\right|
\leq \rho(s)
\eeq
and hence for $s\in(-X(\ell),X(\ell))$ and $\ell\in (0,2\arsinh(1))$ we have
\beq
\label{rho_deriv}
\bigg|\frac{d}{ds}\log\rho(s)\bigg|\leq \frac{\ell}{2\pi\sinh\frac{\ell}{2}}\leq \frac{1}{\pi}.
\eeq
One consequence that we shall use several times is that 
for any $\Lambda>0$ there exists $C\in (0,\infty)$ such that for any $\ell\in (0,2\arsinh(1))$ and $s_0\in (-X(\ell)+\Lambda,X(\ell)-\Lambda)$ (i.e. so that 
$\Cyl_\Lambda(s_0):=(s_0-\Lambda,s_0+\Lambda)\times S^1\subset \Col(\ell)$) if such $s_0$ exists, we have
\beq
\label{rho_equiv}
\frac{1}{C}\rho(s_0)\leq\rho(s)\leq C\rho(s_0)
\eeq
for all $s\in \Cyl_\Lambda(s_0)$.

Because $L_{g}(\{s_0\}\times S^1)=2\pi\rho(s_0)$
we can always bound 
\beq 
\label{est:rho-inj-appendix} \inj_{g}(p_0)\leq \pi\rho(p_0).\eeq
Conversely, \eqref{eq:rho-inj} implies that 
$$\rho(p_0)=\frac{\ell}{2\pi\sinh(\frac\ell2)}\cdot \sinh( \inj_g(p_0))\leq \frac1\pi\sinh( \inj_g(p_0))\leq \frac1\pi\cosh(\inj_g(X(\ell),\theta))\inj_g(p_0)$$
which, once combined with \eqref{inj_at_end}, implies that also the reverse inequality
\beq 
\label{est:rho-by-inj}
\rho(p_0)\leq C\cdot \inj_g(p_0)\eeq
is valid with a universal constant, e.g. with $C=1$, on collars $\Col(\ell)$, 
$0<\ell< 2\arsinh(1)$. 

For $\de\in (0,\arsinh(1))$, the $\de\thin$ part of a collar is given by the subcylinder 
\beq(-X_\de(\ell), X_\de(\ell)) \times S^1 \subseteq \Col(\ell),    \eeq
where 
\beq \label{eq:Xde} 
X_\de(\ell)=  \frac{2\pi}{\ell}\left(\frac{\pi}{2}-\arcsin \left(\frac{\sinh(\frac{\ell}{2})}{\sinh \delta}\right) \right)\eeq
 for $\de\geq \ell/2$, respectively zero for smaller values of $\delta$.

\begin{prop}
\label{X_Xdelta}
There exists universal $C\in (0,\infty)$ such that
for every $\de\in(0,\arsinh(1))$ and $0<\ell\leq 2\de$, we have
$$\frac{\pi}\de -C\leq X(\ell)-X_\de(\ell)\leq \frac{\pi^2}{2\de}.$$
\end{prop}
\begin{proof}
By definition of $X(\ell)$ and $X_\de(\ell)$, we have
$$X(\ell)-X_\de(\ell)=\frac{2\pi}{\ell}\left[
\arcsin \left(\frac{\sinh(\frac{\ell}{2})}{\sinh \delta}\right) 
-\arctan\left(\sinh\left(\frac{\ell}{2}\right)\right) 
\right].$$
Using convexity of $\arcsin:[0,1]\to [0,\frac\pi2]$, we compute
the required upper bound
$$X(\ell)-X_\de(\ell)\leq \frac{2\pi}{\ell}
\arcsin \left(\frac{\sinh(\frac{\ell}{2})}{\sinh \delta}\right)
\leq
\frac{2\pi}{\ell}\frac\pi2\frac{\sinh(\frac{\ell}{2})}{\sinh \delta}
\leq \frac{\pi^2}\ell 
\frac{\frac{\ell}{2}}{\delta}=\frac{\pi^2}{2\de}.$$
On the other hand, by estimating $\arcsin \th\geq \th$ and 
$\arctan\th\leq \th$ for $\th\in [0,1]$, we have
$$X(\ell)-X_\de(\ell)\geq \frac{2\pi}{\ell}
\left[
\frac{\sinh(\frac{\ell}{2})}{\sinh \delta}
-\sinh\left(\frac{\ell}{2}\right)\right]
\geq \pi\left[\frac{1}{\sinh\de}-1\right]
.$$
By estimating $\sinh\de\leq \de+C\de^3$ for $\de\in (0,\arsinh(1))$ and some universal $C$, we have $(\sinh\de)^{-1}\geq \frac{1}{\de}
(1-C\de^2)$ for some possibly different $C$, which completes the lower bound.
\end{proof}

We will use several times that working with respect to the hyperbolic metric, on a collar $\Col$ as above,
\beqa
\label{sizes_on_collars}
|dz^2|=2\rho^{-2};
\qquad
& \|dz^2\|_{L^1(\Col)}=8\pi X(\ell);\\
\qquad
\|dz^2\|_{L^\infty(\Col)}=\frac{8\pi^2}{\ell^2};
\qquad & \|dz^2\|_{L^2(\Col)}^2=\frac{32\pi^5}{\ell^3}-\frac{16\pi^4}{3}+O(\ell^2),
\eeqa
as a short computation verifies.

To analyse sequences of degenerating hyperbolic surfaces we make 
repeated 
use of the differential geometric version of the Deligne-Mumford compactness theorem.

%

\begin{prop}[{cf. \cite[Chapter IV]{Hu}}]  
\label{Mumford}
Let $(M,g_{i},c_{i})$ be a sequence of closed hyperbolic Riemann surfaces of genus $\gamma\geq2$.
Then, after selection of a subsequence,
$(M,g_{i},c_{i})$ converges to a complete hyperbolic punctured Riemann surface $(\Sigma,h,c)$, where $\Si$ is obtained from $M$ by removing a collection 
$\mathscr{E}=\{\sigma^{j}, j=1,...,k\}$ of $k\in\{0,\ldots,3(\ga-1)\}$ pairwise disjoint, homotopically nontrivial, simple closed curves on $M$ and the convergence is as follows:

For each $i$ there exists a collection $\mathscr{E}_{i}=\{\sigma^{j}_{i}, j=1,...,k\}$ of
pairwise disjoint simple closed geodesics on $(M,g_{i},c_{i})$
of length
$\ell(\sigma_{i}^{j})=:\ell_{i}^{j} \rightarrow 0\text{ as }i \rightarrow \infty$,
and an orientation preserving diffeomorphism $F_i:M\to M$ mapping $\si^j$ onto $\si_i^j$,
such that the restriction
$f_i=F_i|_\Si:\Si\rightarrow M\setminus \cup_{j=1}^k\sigma_{i}^{j} $ satisfies
$$(f_i)^{*}g_{i} \rightarrow h \text{ and } (f_i)^{*}c_{i}\to c \text{ in } C_{loc}^{\infty}\text{ on }\Sigma.$$
\end{prop}

{\sc Mathematical Institute, University of Oxford, Oxford, OX2 6GG, UK}

{\sc Mathematics Institute, University of Warwick, Coventry,
CV4 7AL, UK}

\end{document}